\documentclass[onefignum,onetabnum]{siamart220329}
\usepackage{lipsum}
\usepackage{amssymb,amsfonts}
\usepackage{amsmath}
\usepackage{graphicx}
\usepackage{epstopdf}
\usepackage{algorithm,algorithmic}
\usepackage{color}
\usepackage{bm}
\usepackage{booktabs,array}
\usepackage{multirow,multicol}
\usepackage[T1]{fontenc}
\usepackage{geometry}
\usepackage{cite}

\ifpdf
\DeclareGraphicsExtensions{.eps,.pdf,.png,.jpg}
\else
\DeclareGraphicsExtensions{.eps}
\fi

\pdfoutput=1

\newsiamremark{remark}{Remark}
\crefname{remark}{Remark}{Remark}
\newsiamremark{hypothesis}{Hypothesis}
\crefname{hypothesis}{Hypothesis}{Hypotheses}
\newsiamthm{claim}{Claim}
\newsiamremark{assumption}{Assumption}
\crefname{assumption}{Assumption}{Assumption}
\headers{The B-brm-c method for evolutionary PDEs}{W.-H. Luo, X.-M. Gu, B. Carpentieri et al.}

\title{A Bernoulli-barycentric rational matrix collocation method with preconditioning for a class of evolutionary PDEs\thanks{Received by the editors DATE;
\funding{This research was supported by the Scientific Research Fund of Hunan Provincial Science and Technology Department (2022JJ30416) and the Scientific Research Funds of Hunan Provincial Education Department (22A0483). X.-M. Gu is supported by Guanghua Talent Project of Southwestern University of Finance and Economics. The third author is a member of Gruppo Nazionale per il Calcolo Scientifico (GNCS) of Istituto Nazionale di Alta Matematica (INdAM), and this work was partially supported by INdAM-GNCS under Progetti di Ricerca 2023. J. Guo was supported by the Sichuan National Applied Mathematics co-construction project (2022ZX004) and the Scientific Research Foundation (KYTZ202184).}}}

\author{Wei-Hua Luo\thanks{School of Mathematics and Physics,
		Hunan University of Arts and Science, Changde, Hunan 415000, P.R. China 
		(\email{huaweiluo2012@163.com}).}
	\and Xian-Ming Gu\thanks{Corresponding author. School of Mathematics, 
		Southwestern University of Finance and Economics, Chengdu, Sichuan 611130, P.R. China \& Bernoulli Institute for Mathematics, Computer Science and Artificial Intelligence, University of Groningen, Nijenborgh 9, P.O. Box 407, 9700 AK Groningen, The Netherlands
		(\email{guxianming@live.cn}, \email{guxm@swufe.edu.cn}).}
	\and Bruno Carpentieri\thanks{Faculty of Engineering, Free University of Bozen-Bolzano, 39100 Bolzano, Italy 
		(\email{bruno.carpentieri@unibz.it}).}
	\and Jun Guo\thanks{College of Applied Mathematics, Chengdu University of Information Technology, Chengdu, Sichuan 610225, P.R. China
	    (\email{junguo0407@163.com}).}
}

\usepackage{amsopn}

\ifpdf
\hypersetup{
	pdftitle={The B-brm-c method for evolutionary PDEs},
	pdfauthor={W.-H. Luo, X.-M. Gu, B. Carpentieri et al.}
}
\fi

\begin{document}
	


%
%

\maketitle

\begin{abstract}
We propose a Bernoulli-barycentric rational matrix collocation method for two-dimensional evolutionary partial differential equations (PDEs) with variable coefficients that combines Bernoulli polynomials with barycentric rational interpolations in time and space, respectively. The theoretical accuracy $O\left((2\pi)^{-N}+h_x^{d_x-1}+h_y^{d_y-1}\right)$ of our numerical scheme is proven, where $N$ is the number of basis functions in time, $h_x$ and $h_y$ are the grid sizes in the $x$, $y$-directions, respectively, and $0\leq d_x\leq \frac{b-a}{h_x},~0\leq d_y\leq\frac{d-c}{h_y}$. For the efficient solution of the relevant linear system arising from the discretizations, we introduce a class of dimension expanded preconditioners that take the advantage of structural properties of the coefficient matrices, and we present a theoretical analysis of eigenvalue distributions of the preconditioned matrices. The effectiveness of our proposed method and preconditioners are studied for solving some real-world examples represented by the heat conduction equation, the advection-diffusion equation, the wave equation and telegraph equations.
\end{abstract}

\begin{keywords}
evolutionary PDEs; Bernoulli polynomials; barycentric rational interpolation; collocation method; dimension expanded preconditioners.
\end{keywords}
\begin{MSCcodes}
65M70, 65Y05, 65D25
\end{MSCcodes} 

\section{Introduction}
The focus of this paper is the solution of evolutionary PDEs with the form 
\begin{equation}
\begin{cases}
\beta_1\frac{\partial^2 u}{\partial t^2}+\beta_2\frac{\partial u}{\partial t}+\mathcal{L}u=f(t,x,y),& (t,x,y)\in (0,T]\times\Omega =(a,b)\times (c,d),\\
u(0,x,y)=\alpha_0(x,y),~\frac{\partial u(0,x,y)}{\partial t}=\alpha_1(x,y), &(x,y)\in \overline{\Omega}=\Omega\cup \partial\Omega,\\
{\rm Dirichlet~boundary~conditions},
\end{cases}
\label{9yue26_2}
\vspace{-3.5mm}
\end{equation}
where $\beta_1$ and $\beta_2$ are constants, and the spatial operator $\mathcal{L}u=a_{1}\frac{\partial^2u}{\partial x^2}+a_{2}\frac{\partial^2u}{\partial x\partial y}+a_{3}\frac{\partial^2u}{\partial y^2}+a_{4}\frac{\partial u}{\partial x}+a_{5}\frac{\partial u}{\partial y}+a_{6}u$,
where $a_i~(i=1,2,\cdots,6)$ represent known functions of two variables $x, y$, and $u=u(t,x,y)$ is the unknown solution to this equation. This mathematical model is common in various fields of physics and engineering~\cite{pdebook1}. For example, if $\beta_1=0,~\beta_2\neq 0, a_{2}=a_4=a_5=a_6=0,~a_{1}=a_{3}$,
Eq.~(\ref{9yue26_2}) is the heat conduction equation (see, e.g.,~\cite{heat1}); if $\beta_1=0,~\beta_2\neq 0,~a_{2}=0$,
we have the advection/reaction-diffusion equation (see, e.g.,~\cite{advection1,reaction1}).
In the financial field, Eq.~(\ref{9yue26_2}) is sometimes coupled with non-smooth initial-boundary conditions when $\beta_1=0,~\beta_2\neq 0$, see e.g.~\cite{black1,heston1} for
the Heston and Black-Scholes models. Moreover, 
the wave equation (see e.g.~\cite{wave1,wave2,wave3,wave4}) can be obtained from~(\ref{9yue26_2}) if $\beta_1\neq 0,~\beta_2=0, a_{2}=a_4=a_5=a_6=0,~a_{1}=a_{3}$, and telegraph equations~(see, e.g.,~\cite{telegraph1,telegraph2}) if $\beta_1\neq0,~\beta_2\neq 0,~ a_2=a_4=a_5=0, a_1=a_3$. For clarity, we assume that the above PDEs enjoy the well-posedness, existence and uniqueness of the exact solutions throughout
this paper. 

Although there is a wealth of wonderful theory for finding analytic solutions to different types of evolutionary PDEs, the vast majority of PDEs are not amenable to analytic (or closed-form) solutions. Therefore, various numerical methods are commonly used to obtain approximate solutions of Eq.~(\ref{9yue26_2}), including the boundary element method (BEM)~\cite{heat1,wave3}, finite element
method (FEM)~\cite{reaction2,reaction3, numer1}, discontinuous Galerkin method (DGM)~\cite{DG1,DG2,dis-galer1,dis-galer2}, finite difference method (FDM)~\cite{advection2,heat2,Britt18}, finite volume method (FVM)~\cite{volume1, volume2},
spectral method (SM)~\cite{heat3,shen}, Bernoulli matrix method~\cite{Bernoulli1,Bernoulli2,Bernoulli3}, and others. When approximating time-dependent PDEs with initial-boundary value conditions (\ref{9yue26_2}), FEM, DGM, FDM, and FVM always produce time-stepping schemes which require the implementation of sparse matrix-vector products (for explicit schemes) or the solution of a sequence of sparse linear systems (for implicit schemes), but their convergence accuracy is severely constrained by suitable mesh sizes in space and time; this phenomenon becomes more apparent when high accuracy numerical solutions of evolutionary PDEs~(\ref{9yue26_2}) are demanded in some real applications. Furthermore, the convergence and stability of the above four numerical methods, when applied to solve evolutionary PDEs, are only ensured when certain conditions are satisfied for the numerical equation, such as the Courant-Friedrichs-Lewy (CFL) condition for such methods applied to hyperbolic PDEs. Also, their approximating accuracy is frequently compromised when variable coefficients are incorporated into Eq.~(\ref{9yue26_2}). For instance, when utilizing the conventional compact difference scheme to solve the wave equation with variable speed of sound, auxiliary techniques like the Fourier transform in time may be necessary to ensure that the standard fourth-order accuracy in space is attained (see e.g.~\cite{Britt18}). On the other hand, the BEM, SM and Bernoulli matrix methods are known for having very high-order approximation accuracy. However, these methods often generate dense linear systems that can be cumbersome to solve for high-dimensional evolutionary PDEs defined in large time intervals and space domains. The coefficient matrices of these dense systems do not typically exhibit any suitable block structure that could be exploited to develop effective preconditioning methods for accelerating the iterative solution process.
In particular, there are many numerical methods introducing an auxiliary function $v(t,x,y) = u_t$ into second-order evolutionary PDEs, i.e., $\beta_1\neq 0$ in Eq. \eqref{9yue26_2} and then they solve a coupled system of first-order evolutionary PDEs involving the unknown functions $u$ and $v$ via the time-stepping scheme(s), see e.g., \cite{numer1,Sun2009} for details. Such a numerical framework will double the size of the original problem \eqref{9yue26_2} and leads to both the additional computational cost and intricate convergence analyses/conditions.

In this paper, we establish a matrix-collocation method for Eq.~(\ref{9yue26_2}) using combinations of Bernoulli polynomials (BPs) in time and barycentric rational interpolations (BRIs) in space. More specifically, our method can deal with Eq. (\ref{9yue26_2}) in a unified fashion and it does not need to decouple the problem (\ref{9yue26_2}) with $\beta_1\neq 0$. Furthermore, BPs can generate the matrices with sparse structures and favorable numerical properties, and BRIs are easy-to-implement and enjoy high approximating accuracy and good stability properties. As a result, our new matrix-collocation method will have the following characteristics:
\begin{itemize}
\item[(a1)] it is highly accurate in both time and space;
\item[(a2)] it produces linear systems with a well-structured coefficient matrix that can be easily transformed into a block matrix for which an efficient preconditioner may be developed and analyzed.
\end{itemize}

The remainder of this paper is organized as follows. Section~\ref{sec:prelim} presents basic preliminaries on BPs and BRIs. Section~\ref{sec:Bernoulli-barycentric} introduces the matrix-collocation method for Eq.~(\ref{9yue26_2}) and presents some theoretical results on numerical accuracy, to demonstrate the effectiveness of the proposed scheme. In Section \ref{sec:precon}, a class of new dimension expanded preconditioners is developed to take advantage of the structural properties of the coefficient matrices of pertinent linear systems arising from the discretization. Section \ref{sec:numeri} provides several numerical examples to assess the practical performance of the new matrix-collocation method and of the proposed preconditioners for solving the heat conduction equation, the advection-diffusion equation, the wave equation and telegraph equations. Finally, Section \ref{sec:conclusions} provides a summary of conclusions drawn from this study.
\section{Preliminaries}
\label{sec:prelim}
Preliminaries on the BPs and BRIs can be found in \cite{Bernoulli1, Bernoulli2, Bernoulli3} and \cite{bary1,bary2,bary3}, respectively. The BPs $B_n(t) ~(n=0,1,\cdots)$ can be generated using the recurrence formulae:
\begin{equation}
\begin{cases}
B_0(t)=1,\\
B'_n(t)=nB_{n-1}(t),~~\int^1_0B_n(s)ds=0, & n\geq 1,
\end{cases}
\label{10yue2_1}
\end{equation}
which help us immediately derive the equality:
\begin{equation}
\int^t_0[B_0(s),B_1(s),\cdots,B_N(s)]ds=[B_0(t),B_1(t),\cdots,B_N(t)]P
+\frac{B_{N+1}(t)-B_{N+1}(0)}{N+1}e^{\top}_{N+1},
\label{10yue2_2}
\end{equation}
with
\begin{equation*}
P=\begin{bmatrix}
	-B_1(0) &-\frac{B_2(0)}{2}&\cdots &-\frac{B_N(0)}{N} &0\\
	1 &0  &\cdots & 0 & 0\\
	0  & \frac{1}{2}  & \ddots & 0 & 0\\
    \vdots & \vdots & \ddots & \ddots & \vdots \\
    0  & 0  & \cdots &\frac{1}{N} & 0
\end{bmatrix},~~{\rm and}~~
e_{N+1}=\begin{bmatrix}
0\\
\vdots\\
0\\
1
\end{bmatrix}.
\end{equation*}

The following properties are obtained from \cite{Bernoulli2,Bernoulli4}.
\begin{proposition}{\rm(\hspace{-0.2mm}\cite{Bernoulli2,Bernoulli4})} The BPs $B_n(t)~(n=1,2,\cdots)$ satisfy the inequality  $||B_n(t)||_{\infty}\leq C n!(2\pi)^{-n},~(n=1,2,\cdots),$
where $C$ is a constant independent of $n$.
\label{pro1}
\end{proposition}

\begin{proposition}{\rm(\hspace{-0.2mm}\cite{Bernoulli2})}
\label{pro2}
If $g(t)\in L^2[0,1]$ is a sufficiently smooth function, and  $\sum^{\infty}_{n=0}g_nB_n(t)$ is an approximation of $g(t)$ derived using the Bernoulli series~(\ref{10yue2_1}), then
\begin{equation}	\label{10yue2_3}
g_n=\frac{1}{n!}\int^1_0g^{(n)}(t)dt.
\end{equation}
\end{proposition}

\begin{proposition}{\rm(\hspace{-0.2mm}\cite{Bernoulli2})}
\label{pro3}
Let $g(t)\in L^2[0,1]$ be a sufficiently smooth function, and suppose $g_N(t)=\sum^{N}_{n=0}g_nB_n(t)$ is the approximation of $g(t)$ obtained by truncating the Bernoulli series (\ref{10yue2_1}), specifically, $g(t)=g_N(t)+E(g_N(t))$. Then, there holds
\begin{equation}
||E(g_N(t))||_{\infty}\leq CG(2\pi)^{-N},~t\in [0,1],
\label{10yue2_4}
\end{equation}
where $C$ is a positive constant, and $G$  is a bound for all the derivatives of the function $g(t)$,~(i.e. $G\geq ||g^{(i)}(t)||_{\infty},~i=0,1,2,\cdots $).
\end{proposition}

Next, to approximate the unknown solution in space, we introduce BRIs, which have been 
a subject of extensive research (see, e.g.~\cite{bary1,bary2,bary3}). Consider the uniform partitioning $a=x_0<\cdots< x_r=b$ of an interval $[a,b]$ with grid size $h=\frac{b-a}{r}$, and an integer $\tilde{d}$ with $0\leq \tilde{d}\leq r$. We denote by $J_m=\{s\in J | m-\tilde{d}\leq s\leq m\}$, where $J=\{0,1,2,\cdots, r-\tilde{d}\}$ represents an index set. Then, basis functions for BRI can be defined as
\begin{equation}
\varphi_m(x)=\frac{\frac{\omega_m}{x-x_m}}{\sum^{r}_{k=0}\frac{\omega_k}{x-x_k}},~~\omega_m=\sum_{s\in J_m}(-1)^s\prod^{s+\tilde{d}}_{j=s,j\neq m}\frac{1}{x_m-x_j}, ~~m=0,1,\cdots,r.
\label{4yue9_1}
\end{equation}
The following results for the functions $\varphi_i(x)$ and for $0\leq i,~j\leq r$ are demonstrated in~\cite{bary4}:
\[
\begin{gathered}
\varphi'_j(x_i)=\frac{\omega_j/\omega_i}{x_i-x_j},~j\neq i;~~~\varphi'_i(x_i)=-\sum_{j\neq i}\varphi'_j(x_i). \hfill \\
\varphi''_j(x_i)=\frac{-2\omega_j/\omega_i}{x_i-x_j}\left(\sum_{k\neq i}\frac{\omega_k/\omega_i}{x_i-x_k}+\frac{1}{x_i-x_j}\right),~j\neq i;~~~
\varphi''_i(x_i)=-\sum_{j\neq i}\varphi''_j(x_i).
\end{gathered}
\]

Two interpolation error results that will be utilized in the analysis and evaluation of the numerical method proposed in Section~\ref{sec:Bernoulli-barycentric} are described in the next two properties.
\begin{proposition}{\rm(\hspace{-0.2mm}\cite[Theorem 2]{bary1})}
\label{pro5}
Suppose $u(x)\in C^{\tilde{d}+2}[a,b],~\tilde{d}\geq 1$. There holds
\begin{equation*}
\|u_b(x)-u(x)\|_{\infty}\leq\begin{cases}
 h^{\tilde{d}+1}(b-a)\frac{\|u^{(\tilde{d}+2)}\|_{\infty}}{\tilde{d}+2},&\text{if}~r-\tilde{d}~\text{is odd};\\
h^{\tilde{d}+1}(b-a)\frac{\|u^{(\tilde{d}+2)}\|_{\infty}}{\tilde{d}+2}+h^{\tilde{d}+1}
\frac{\|u^{(\tilde{d}+1)}\|_{\infty}}{\tilde{d}+1},&\text{if}~r-\tilde{d} ~\text{is even}.
\end{cases}
\end{equation*}
\end{proposition}
\begin{proposition}{\rm(\hspace{-0.2mm}\cite[Theorem 3, Theorem 5]{bary2})}
\label{pro6}{If $u(x)\in C^{\tilde{d}+3}[a,b],~\tilde{d}\geq 2$, then $\|(u_b(x)-u(x))'\|_{\infty}\leq Ch^{\tilde{d}}$. If $u(x)\in C^{\tilde{d}+4}[a,b], ~\tilde{d}\geq 3$, there is a constant $C$ that depends only on $b-a$ and $u(x)$ such that $\|(u_b(x)-u(x))''\|_{\infty}\leq Ch^{\tilde{d}-1}$.}
\end{proposition}
\section{The Bernoulli-barycentric rational matrix-collocation method}
\label{sec:Bernoulli-barycentric}
This section begins with an introduction to the conventional Bernoulli matrix method (shortly abbreviated as {\em Bm}) for standard ordinary differential equations (ODEs) with initial
values. Then, the method is combined with BRIs
to derive the Bernoulli-barycentric rational matrix-collocation method
(shortly, {\em B-brm-c}) for Eq.~(\ref{9yue26_2}). Since BPs are defined in the interval $[0,1]$, in the remaining part of this paper we will study them in this interval. However, they can be easily translated to
any interval $[0,T]$ by a simple linear mapping.

\subsection{The {\em Bm} method for ODE}
\label{sec:Bernoulli1}
Consider the ODE of the form
\begin{equation}\label{10yue3_1}
\begin{cases}
\beta_1u''(t)+\beta_2u'(t)+\kappa u(t)=f(t),& t\in(0,1],\\
u(0)=\alpha_0,~u'(0)=\alpha_1.
\end{cases}
\end{equation}
By twice integrating both sides of the above ODE (\ref{10yue3_1}) from 0 to $t$, and considering the initial values, the equivalent integral form
\begin{equation}\label{10yue3_2}
\beta_1u(t)+\beta_2\int^t_0u(s)ds+\kappa \int^{t}_0\int^{\xi}_0u(s)dsd\xi=F(t)
\end{equation}
is obtained, where $F(t)=\int^{t}_0\int^{\xi}_0f(s)dsd\xi+(\alpha_0\beta_2+\alpha_1\beta_1)t+\alpha_0\beta_1$.

Denote as $u_B(t)=\sum^N_{n=0}B_n(t)u_n$ the truncated Bernoulli series used to approximate $u(t)$ in $[0,1]$. Substituting $u_B(t)$ into (\ref{10yue3_2}) and approximating $F(t)$ with $F_B(t)=\sum^N_{n=0}B_n(t)F_n$, we get
\begin{equation}\label{10yue3_3}
\beta_1\sum^N_{n=0}B_n(t)u_n+\beta_2\int^t_0\sum^N_{n=0}B_n(s)u_nds
+\kappa \int^t_0\int^{\xi}_0\sum^N_{n=0}B_n(s)u_ndsd\xi=\sum^N_{n=0}B_n(t)F_n,
\end{equation}
where $F_n=\frac{1}{n!}\int^1_0F^{(n)}(s)ds$ are obtained by applying Eq.~(\ref{10yue2_3}) from Proposition~\ref{pro2}. Denoting by
\[B(t)=[B_0(t),B_1(t),\cdots,B_N(t)], \ U=[u_0,u_1,\cdots,u_N]^\top,\ F=[F_0,F_1,\cdots,F_N]^\top,\]
Eq.~(\ref{10yue3_3}) can be rewritten as
\begin{equation}\label{10yue4_1}
\beta_1B(t)U+\beta_2\int^t_0B(s)dsU +\kappa \int^t_0\int^{\xi}_0B(s)dsd\xi U=B(t)F.
\end{equation}
At this stage, we use the following estimates from~(\ref{10yue2_2}), 
\begin{equation}\label{1yue5_1}
 \int^t_0B(s)ds\approx B(t)P~\text{and}~\int^t_0\int^{\xi}_0B(s)dsd\xi\approx B(t)P^2,
\end{equation}
into~(\ref{10yue4_1}) to obtain the approximation equation
\begin{equation}\label{10yue4_2}
\beta_1B(t)U+\beta_2 B(t)PU+\kappa B(t)P^2U=B(t)F ~\Leftrightarrow
~B(t)\left[(\beta_1I+\beta_2 P+\kappa P^2)U - F\right] = 0.
\end{equation}
Since the Bernoulli basis is complete, we can eliminate $B(t)$ in~Eq. (\ref{10yue4_2}) resulting  in the system of linear equations
\begin{equation}\label{10yue4_3}
(\beta_1I+\beta_2 P+\kappa P^2)U=F,
\end{equation}
 where $I$ is the identity matrix of order $N+1$ and $U$ is the unknown vector to be solved. The following theorem describes the truncation error that exists between the original equation~(\ref{10yue3_2}) and the approximation equation~(\ref{10yue4_2}).
\begin{theorem}\label{th1}
The truncation error term between (\ref{10yue3_2}) and (\ref{10yue4_2}) does not exceed $C(2\pi)^{-N}$, where $C$ is a constant that only depends on $\beta_1,~\beta_2,~\kappa$ and the bounds for the derivatives of functions $u(t),~F(t)$.
\end{theorem}
\begin{proof}
First, by applying Proposition~\ref{pro3} to Eq. (\ref{10yue3_2}), we can establish that there holds
\begin{equation}\label{11yue21_1}
\begin{split}
&\beta_1(u_B(t)+E(u_B(t)))+\beta_2 \left(\int^t_0u_B(s)ds+\int^t_0E(u_B(s))ds\right)
+\kappa \int^t_0\int^{\xi}_0u_B(s)dsd\xi\\
&+\kappa \int^t_0\int^{\xi}_0E(u_B(s))dsd\xi
=F_B(t)+E(F_B(t)),
\end{split}
\end{equation}
where $||E(u_B(t))||_{\infty}\leq C_1G_1(2\pi)^{-N}$, $||E(F_B(t))||_{\infty}\leq C_2G_2(2\pi)^{-N}$, with $G_1\geq ||u^{(i)}(t)||_{\infty}$ and $G_2\geq ||F^{(i)}(t)||_{\infty},~(i=0,1,2,\cdots)$. Namely, we have
\begin{equation}\label{10yue4_4}
\begin{split}
\beta_1B(t)U+\beta_2 \int^t_0B(s)dsU+\kappa \int^t_0\int^{\xi}_0B(s)dsd\xi U
&=B(t)F-\beta_1E(u_B(t))-\beta_2 \int^t_0 E(u_B(s))ds\\
&\quad~-\kappa\int^t_0\int^{\xi}_0E(u_B(s))dsd\xi+E(F_B(t)).
\end{split}
\end{equation}

Moreover, by the equality (\ref{10yue2_2}), we know that the difference between $\int^t_0B(s)ds$ and $B(t)P$
is $\frac{B_{N+1}(t)-B_{N+1}(0)}{N+1}e^\top_{N+1},$  and that
\[\int^t_0\int^{\xi}_0B(s)dsd\xi=B(t)P^2+\frac{B_{N+1}(t)-B_{N+1}(0)}{N+1}e^\top_{N+1}P
+\int^t_0\frac{B_{N+1}(\xi)-B_{N+1}(0)}{N+1}e^\top_{N+1}d\xi.\]
Hence, we can rewrite Eq.~(\ref{10yue4_4}) as
\begin{equation}\label{10yue4_5}
\begin{split}
\beta_1B(t)U+\beta_2 B(t)PU +\kappa B(t)P^2U=B(t)F-\beta_1E(u_B(t))-\beta_2 \int^t_0 E(u_B(s))ds \\
-\kappa\int^t_0\int^{\xi}_0E(u_B(s))dsd\xi
 +E(F_B(t))-\beta_2u_N\frac{B_{N+1}(t)- B_{N+1}(0)}{N+1}e_{N+1}^\top\\ 
 -\kappa\frac{B_{N+1}(t)-B_{N+1}(0)}{N+1}e^\top_{N+1}PU
-\kappa\int^t_0\frac{B_{N+1}(\xi)-B_{N+1}(0)}{N+1}e^\top_{N+1}d\xi U.
\end{split}
\end{equation}
Additionally, we know from Propositions~\ref{pro1}-\ref{pro2} that
\[\left\|u_N\frac{B_{N+1}(t)-B_{N+1}(0)}{N+1}e_{N+1}^T\right\|_{\infty}\leq
\frac{2C(N+1)!(2\pi)^{-N-1}}{N+1}\cdot \frac{1}{N!}\int^1_0u^{(N)}(s)ds
\leq C_3(2\pi)^{-N},\]
\begin{equation*}
\begin{split}\left\|\frac{B_{N+1}(t)-B_{N+1}(0)}{N+1}e_{N+1}^TPU\right\|_{\infty}
&=\left\|\frac{B_{N+1}(t)-B_{N+1}(0)}{N+1}\frac{u_{N-1}}{N}\right\|_{\infty}\\ 
&\leq \frac{2C(N+1)!(2\pi)^{-N-1}}{N+1}\cdot \frac{1}{N!}\int^1_0u^{(N-1)}(s)ds
\leq C_4(2\pi)^{-N},
\end{split}
\end{equation*}
\[\left\|\int^t_0\frac{B_{N+1}(\xi)-B_{N+1}(0)}{N+1}e_{N+1}^TUd\xi\right\|_{\infty}
=\left\|u_N\int^t_0\frac{B_{N+1}(\xi)-B_{N+1}(0)}{N+1}d\xi\right\|_{\infty}
\leq C_3(2\pi)^{-N},\]
where $C_3$ and $C_4$ are constants that depend on the bound of $u^{(N)}(t),~u^{(N-1)}(t),~t\in [0,1]$, respectively. Hence, the difference between (\ref{10yue4_5}) and (\ref{10yue4_2}) does not exceed
\[(\beta_1C_1G_1+\beta_2C_1G_1+\kappa C_1G_1+C_2G_2+\beta_2C_3+\kappa C_4+\kappa C_3)(2\pi)^{-N},\]
and this completes the proof of Theorem \ref{th1}.  
\end{proof}
\begin{remark} From Eq.~(\ref{10yue4_3}), we see that the coefficient matrix has the following structure:
\[\beta_1I+\beta_2P+\kappa P^2=\begin{bmatrix}
	*  & * &*&*&\cdots & *& *& *\\
	*  & *\ &*&*&\cdots & * &*& *\\
    * &* &*&& &  & & \\
      & \ddots  &\ddots & \ddots & \\
      & &\ddots  &\ddots & \ddots & \\
      & &  &\ddots  &\ddots & \ddots & \\
      & & & &\ddots  &\ddots & \ddots & \\
	  &  && &  &*& * & *
\end{bmatrix}.\]
Applying the {\em Bm} method to approximate the solution $u(t,x,y)$ of Eq.~(\ref{9yue26_2}) results in the block-structured system of linear equations. The matrix structure of $\beta_1I+\beta_2P+\kappa P^2$ can be utilized to develop efficient preconditioners, as shown in the next section.
\end{remark}
\subsection{The {\em B-brm-c} method}
\label{sec:Bbrmc1}
In this section, we solve PDEs of the form~(\ref{9yue26_2}) using BPs in time and BRIs in space. By twice integrating both sides of evolutionary PDEs~(\ref{9yue26_2}) from 0 to time $t$, and employing the initial conditions, we can obtain
\begin{equation}\label{10yue5_2}
 \beta_1u(t,x,y)+\beta_2\int^t_0u(s,x,y)ds+
 \int^t_0\int^{\xi}_0\mathcal{L}u(s,x,y)dsd\xi=F(t,x,y),~
 (t,x,y)\in (0,1]\times\Omega,
\end{equation}
 where $F(t,x,y)=\int^t_0\int^{\xi}_0f(s,x,y)dsd\xi+[\beta_2\alpha_0(x,y)
 +\beta_1\alpha_1(x,y)]t+\beta_1\alpha_0(x,y)$.

Let \[\tilde{u}(t,x,y)=\sum^N_{n=0}B_n(t)u^n(x,y)
 ~~\text{and}~~ \tilde{F}(t,x,y)=\sum^N_{n=0}B_n(t)F^n(x,y)\]
  be  approximations of $u(t,x,y)$ and $F(t,x,y)$, respectively, where $u^n(x,y),~(n=0,1,\cdots,N)$ are all unknown functions of $x,y$, and $F^n(x,y)=\frac{1}{n!}\int^1_0F^{(n)}(s,x,y)ds,~(n=0,1,\cdots,N)$.
Substituting these expressions into (\ref{10yue5_2}) results in
\begin{equation}\label{10yue5_3}
\begin{array}{ll}
 \beta_1\sum\limits^{N}_{n=0}B_n(t)u^n(x,y)+\beta_2\sum\limits^{N}_{n=0}\left(\int^t_0B_n(s)dsu^n(x,y)\right)+a_{1}\sum\limits^{N}_{n=0}\left(
 \int^t_0\int^{\xi}_0B_n(s)dsd\xi\frac{\partial^2u^n(x,y)}{\partial x^2}\right)\\
 +~a_{2}\sum\limits^{N}_{n=0}\left(
\int^t_0\int^{\xi}_0B_n(s)dsd\xi\frac{\partial^2u^n(x,y)}{\partial x\partial y}\right)
 +~a_{3}\sum\limits^{N}_{n=0}\left(\int^t_0\int^{\xi}_0B_n(s)dsd\xi\frac{\partial^2u^n(x,y)}{\partial y^2}\right)\\
 +~a_{4}\sum\limits^{N}_{n=0}\left(\int^t_0\int^{\xi}_0B_n(s)dsd\xi\frac{\partial u^n(x,y)}{\partial x}\right)
 +~a_{5}\sum\limits^{N}_{n=0}\left(\int^t_0\int^{\xi}_0B_n(s)dsd\xi\frac{\partial u^n(x,y)}{\partial y}\right)\\
 +~a_{6}\sum\limits^{N}_{n=0}\left(\int^t_0\int^{\xi}_0B_n(s)dsd\xi u^n(x,y)\right) =\sum\limits^{N}_{n=0}B_n(t)F^n(x,y)+O\left((2\pi)^{-N}\right),
 ~(t,x,y)\in (0,1]\times \Omega.
 \end{array}
\end{equation}
Let $U(x,y)=[u^0(x,y),~u^1(x,y),~\cdots,~u^N(x,y)]^\top$, $F(x,y)=[F^0(x,y),~F^1(x,y),~\cdots,~F^N(x,y)]^\top$,
\[\frac{\partial^2 U(x,y)}{\partial x^2}=\left[\frac{\partial^2u^0(x,y)}{\partial x^2},~\frac{\partial^2u^1(x,y)}{\partial x^2},~\cdots,~\frac{\partial^2u^N(x,y)}{\partial x^2}\right]^\top,\]
and $\frac{\partial^2 U(x,y)}{\partial x\partial y}, \frac{\partial^2 U(x,y)}{\partial y^2}, \frac{\partial U(x,y)}{\partial x},\frac{\partial U(x,y)}{\partial y}$
 be defined in the same fashion. As in the previous section, we approximate $\int^t_0B(s)ds\approx B(t)P$ and $\int^t_0\int^{\xi}_0B_n(s)dsd\xi\approx B(t)P^2$ with the help of Eq.~(\ref{1yue5_1}).
Then, Eq.~(\ref{10yue5_3}) can be expressed as
\begin{equation}\label{10yue5_4}
 \beta_1B(t)U(x,y)+\beta_2B(t)PU(x,y)
 +B(t)P^2\mathcal{L}U(x,y)
 =B(t)F(x,y)+O\left((2\pi)^{-N}\right).
\end{equation}
Consider the uniform partitioning $a=x_0<\cdots< x_{\tilde{M}}=b, ~c=y_0<\cdots<y_{\overline{M}}=d$ with grid sizes $ h_x=\frac{b-a}{\tilde{M}},~ h_y=\frac{d-c}{\overline{M}}$, and fix two positive integers $d_x,~d_y$ with $0\leq d_x\leq \tilde{M},~0\leq d_y\leq \overline{M}$. We approximate $u^n(x,y)$ as \[\tilde{u}^n(x,y)=\sum^{\tilde{M},\overline{M}}_{i=0,j=0}\varphi_{i}(x)
\varphi_j(y)u^n_{i,j}\]
for each $0\leq n\leq N$,~where $u^n_{i,j}=u(t_n,x_i,y_j)$ and functions $\varphi_{i}(x)$ (resp., $\varphi_{j}(y)$) are defined according to~(\ref{4yue9_1}) using $x_0,\cdots,x_{\tilde{M}}$ and $\tilde{d}=d_x$ (resp., $y_0,\cdots,y_{\overline{M}}$ and $\tilde{d}=d_y$). At this stage, the following collocation equations 
\begin{equation}\label{1yue19_1}
\begin{array}{ll}
 \beta_1B(t)\tilde{U}(x,y)+\beta_2B(t)P\tilde{U}(x,y)
 +B(t)P^2\mathcal{L}\tilde{U}(x,y)\\
 =B(t)F(x,y)+O\left((2\pi)^{-N}+h_x^{d_x-1}+h_y^{d_y-1}\right),
 ~~(t,x,y)\in (0,1]\times  \Omega,
 \end{array}
\end{equation}
can be derived from~(\ref{10yue5_4}), with $\tilde{U}(x,y)=[\tilde{u}^0(x,y),~\tilde{u}^1(x,y),~\cdots,~\tilde{u}^N(x,y)]^\top.$ 

From Propositions~\ref{pro5}--\ref{pro6}, the error approximation $O\left(h_x^{d_x-1}+h_y^{d_y-1}\right)$ is obtained. By considering the completeness of the Bernoulli basis $B(t)$, ignoring the error  term $O\left((2\pi)^{-N}+h_x^{d_x-1}+h_y^{d_y-1}\right)$ and using the boundary conditions in~(\ref{9yue26_2}), we finally derive the linear equations
\begin{equation}\label{10yue5_5}
\begin{array}{ll}
HU=R,~~H=\beta_1I\otimes I_{\hat{k}}+\beta_2P\otimes I_{\hat{k}}+
P^2\otimes Q,
\end{array}
\end{equation}
where $\hat{k}=(\tilde{M}-1)(\overline{M}-1)$, $Q=A_{1}H_{1}+ A_{2}H_{2}+A_{3}H_{3}
+A_{4}H_{4}+A_{5}H_{5}+ A_{6}$,
\[U=[\mathbf{U}^0,\mathbf{U}^1,\cdots,\mathbf{U}^N]^\top=[\underbrace{u^0_{1,1}, \cdots, u^0_{1,\overline{M}-1},\cdots,u^0_{\tilde{M}-1,1}, \cdots, u^0_{\tilde{M}-1,\overline{M}-1}}_{\hat{k}},\cdots,\] \[\cdots,\underbrace{u^N_{1,1}, \cdots, u^N_{1,\overline{M}-1},\cdots,u^N_{\tilde{M}-1,1}, \cdots, u^N_{\tilde{M}-1,\overline{M}-1}}_{\hat{k}}]^\top,\]
\[
A_{i}=\mathrm{diag}\Big(a_{i}(x_1,y_1),\cdots,a_{i}(x_1,y_{\overline{M}-1}),\cdots,
a_{i}(x_{\tilde{M}-1},y_1),\cdots,a_{i}(x_{\tilde{M}-1},
y_{\overline{M}-1})\Big),~i=1,\cdots,6,
\]
\[H_1=T_{xx}\otimes Q_y,~H_2=T_{x}\otimes T_y,~H_3=Q_{x}\otimes T_{yy},~H_4=T_{x}\otimes Q_y,~H_5=Q_{x}\otimes T_y,\]
with two identity matrices $Q_x$ and $Q_y$ of orders $\tilde{M}-1$ and $\overline{M}-1$, respectively. We also define 
\[T_{x}=\begin{bmatrix}
 \varphi_1'(x_1)&\cdots&\varphi_{\tilde{M}-1}'(x_1)\\
\vdots&\vdots&\vdots\\ \varphi_1'(x_{\tilde{M}-1})&\cdots&\varphi_{\tilde{M}-1}'(x_{\tilde{M}-1})
\end{bmatrix},~T_{xx}=\begin{bmatrix}
 \varphi_1''(x_1)&\cdots&\varphi_{\tilde{M}-1}''(x_1)\\
\vdots&\vdots&\vdots\\ \varphi_1''(x_{\tilde{M}-1})&\cdots&\varphi_{\tilde{M}-1}''(x_{\tilde{M}-1})
\end{bmatrix}\in\mathbb{R}^{(\tilde{M}-1)\times (\tilde{M}-1)},\]
and $T_y, T_{yy}$ are obtained by simply substituting $x_i~(i=1,2,\cdots,\tilde{M}-1)$ in $T_x, T_{xx}$ with $y_i~(i=1,2,\cdots,\overline{M}-1)$.


Using Propositions~\ref{pro5} and~\ref{pro6}, we get the following estimate of the truncation error.
\begin{theorem}\label{th2}
The truncation error between (\ref{10yue5_2}) and (\ref{1yue19_1}) is  $O\left((2\pi)^{-N}+h_x^{d_x-1}+h_y^{d_y-1}\right)$.
 \end{theorem}

For the special case $\beta_1=0,~\beta_2=1$ in Eq.~(\ref{9yue26_2}),
we examine the existence of the solution of~(\ref{10yue5_5}). By means of 
Sections \ref{sec:Bernoulli1}--\ref{sec:Bbrmc1}, the coefficient matrix $H$ in~(\ref{10yue5_5}) can be clearly represented as 
\begin{equation}\label{10yue29_1}
H=I\otimes I_{\hat{k}}+P\otimes Q=\begin{bmatrix}
	I-B_1(0)Q  & -\frac{B_2(0)Q}{2} & -\frac{B_3(0)Q}{3} & \cdots & -\frac{B_N(0)Q}{N} & \mathbf{0}\\
	Q  & I  &\mathbf{0}& \cdots & \mathbf{0} & \mathbf{0}\\
	\mathbf{0}  & \frac{Q}{2}  & I  & \cdots & \mathbf{0} & \mathbf{0}\\
    \vdots & \vdots & \ddots &\ddots &\ \vdots &\ \vdots \\
    \mathbf{0}  & \mathbf{0} & \mathbf{0}  &\ddots& I &\mathbf{0}\\
    \mathbf{0} & \mathbf{0} & \mathbf{0}  & \cdots & \frac{Q}{N} & I
\end{bmatrix},
\end{equation}
as we only need to integrate both sides of Eq.~(\ref{9yue26_2}) once from 0 to $t$. And from~(\ref{10yue29_1}), the following result can be proved.
\begin{theorem}	\label{th3}
The coefficient matrix $H$ in Eq. (\ref{10yue29_1}) is nonsingular if and only if $T_1=\sum^{N}_{j=0}(-1)^j\frac{B_j(0)}{j!}Q^j$ is nonsingular. In particular, as $N\rightarrow +\infty$, the nonsingularity of $H$ is equivalent to that of $Q$ in (\ref{10yue29_1}).
\end{theorem}
\begin{proof} When we define
 \[S_1=\begin{bmatrix}
 I&&&&\\ -Q&I&&&\\&&I&&\\ &&&\ddots&\\ &&&&I
 \end{bmatrix},
 ~S_2=\begin{bmatrix}
 I&&&&\\ &I&&&\\&-\frac{Q}{2}&I&&\\ &&&\ddots&\\ &&&&I
 \end{bmatrix},~\cdots,
 ~S_N=\begin{bmatrix}
 I&&&&\\ &I&&&\\&&I&&\\ &&&\ddots&\\ &&&-\frac{Q}{N}&I
 \end{bmatrix},\]
it immediately follows
\begin{equation}\label{1yue9_1}
H\cdot (S_N\cdot S_{N-1}\cdots S_1)=T,\ \ \
T=\begin{bmatrix}
T_1&T_2&\cdots&T_N&\mathbf{0}\\&I&&&\\&&\ddots&&\\&&&I&\\&&&&I
\end{bmatrix},
\end{equation}
where \[T_j=-\frac{B_j(0)Q}{j}+\frac{B_{j+1}(0)Q^2}{j(j+1)}
-\frac{B_{j+2}(0)Q^3}{j(j+1)(j+2)}+\cdots
+(-1)^{N+j-1}\frac{B_N(0)Q^{N-j+1}}{j(j+1)\cdots N},~j=2,\cdots,N.\]
Consequently, the invertibility of $H$ is equivalent to that of $T_1$.
On the other hand, we suppose that the matrix $Q$ has 
$k$ different eigenvalues: $\lambda_1,~\lambda_2,~\cdots,~\lambda_k$, then, 
		by the theory of Jordan matarix and matrix function,  there holds
		\[Q=DJD^{-1},~J=\mathrm{diag}(J_1,J_2,\cdots,J_k),~~J_i=\begin{bmatrix}
			\lambda_i&1&&\\ &\ddots&\ddots&\\
			&&\lambda_i&1\\&&&\lambda_i	
		\end{bmatrix}, ~(i=1,~2,~,k)\]
		and 
		\[
		e^{-Q}-I=D\cdot \mathrm{diag}(e^{-J_1}-I,\cdots,e^{-J_k}-I)\cdot D^{-1},~
		e^{-J_i}-I=\begin{bmatrix}
			e^{-\lambda_i}-1&*&\cdots&*\\
			&e^{-\lambda_i}-1&\ddots&\vdots\\
			&&\ddots&*\\
			&&&e^{-\lambda_i}-1
		\end{bmatrix},\]
		thus, we can clearly see that the nonsingularity of $Q$ is the same as that of $e^{-Q}-I$. 
		
The following series expansion related to BPs \cite{Bernoulli1}
		\[\frac{se^{st}}{e^s-1}=\sum^{+\infty}_{j=0}\frac{B_j(t)}{j!}s^j\]
helps us to obtain 
		\[\lim_{N\rightarrow +\infty}T_1=\lim_{N\rightarrow +\infty}\sum^{N}_{j=0}(-1)^j\frac{B_j(0)}{j!}Q^j=-Q(e^{-Q}-I)^{-1}, \]
		and $Q$ is nonsingular if and only if $\lim_{N\rightarrow +\infty}T_1$ is nonsingular. This completes the proof of Theorem~\ref{th3}. 
\end{proof}
\begin{remark} 
In practice, although we cannot directly use Theorem~\ref{th3} to verify the invertibility of the coefficient matrix $H$ in Eq. (\ref{10yue29_1}) due to that a small integer $N$ is selected, but the uniformly discretized linear systems (\ref{10yue5_5}) including Eq. \eqref{10yue29_1} still have the unique solution in all our numerical implementations. Additionally, according to Theorem~\ref{th3}, the direct solution of linear system~($\ref{10yue29_1}$) may not be feasible due to the complex representation of the matrix $T_1$ (not sparse) required for matrix factorization~(\ref{1yue9_1}). This motivates us to develop the efficient preconditioned iterative solver(s) for the discretized linear system~(\ref{10yue5_5}).

%
\end{remark}
\section{A class of dimension expanded preconditioners for discretized linear systems}
\label{sec:precon}
In this section, we study a family of preconditioners for the fast iterative solution of linear system~(\ref{10yue5_5}), computed by expanding the dimension of the coefficient matrix following an approach presented in~\cite{luo2,luo3}. In this section, we consider that $N$ is a positive definite even number since in this case there is $B_N(0)\neq 0$, which allows the suitable factorization of these preconditioners. 
\subsection{Case 1: $\beta_1=0$ and $\beta_2=1$}
\label{sec:precon1}
In this case, the coefficient matrix in~(\ref{10yue5_5}) is represented by matrix $H$ in~(\ref{10yue29_1}). Initially, we multiply a matrix $\tilde{S}$ to transform Eq. (\ref{10yue5_5}) into the equivalent system
\begin{equation}\label{10yue30_1}
\tilde{S}HU = \tilde{S}R \Longleftrightarrow \tilde{H}U=\tilde{R},
\end{equation}
where
\begin{equation*}
\tilde{H}=\begin{bmatrix}
	I & B_1(0)I & B_2(0)I &\cdots &B_{N-1}(0)I& B_N(0)I \\
	Q  & I  & \mathbf{0} & \cdots & \mathbf{0}& \mathbf{0}\\
	\mathbf{0} & \frac{Q}{2} & I  & \cdots  & \mathbf{0}&\mathbf{0}\\
    \vdots & \vdots & \ddots & \ddots &\vdots & \vdots \\
    \mathbf{0} &\mathbf{0} & \mathbf{0} & \ddots & I& \mathbf{0}\\
    \mathbf{0} &\mathbf{0} & \mathbf{0} & \cdots & \frac{Q}{N} &I
\end{bmatrix}~{\rm and}~\tilde{S}=\begin{bmatrix}
I&B_1(0)I&\cdots&B_N(0)I\\
&I&&\vdots\\
&&\ddots&\\
&&&I
\end{bmatrix}.
\end{equation*}
Next, we augment Eq. (\ref{10yue30_1}) into the new system
\begin{equation}\label{10yue30_3}
\mathbb{H}\mathbb{U}=\mathbb{R},
\end{equation}
where $\mathbb{U}=[-\mathbf{U}^N,\mathbf{U}^0,\mathbf{U}^1,\cdots,\mathbf{U}^N]^\top$, and
\[\mathbb{H}=\begin{bmatrix}
	I&I &B_1(0)I &B_2(0)I &\cdots &B_{N-1}(0)I &I+B_{N}(0)I\\
	\mathbf{0}&Q & I &\mathbf{0} &\cdots&\mathbf{0}&\ \mathbf{0}\\
	\mathbf{0}&\mathbf{0} & \frac{Q}{2}  & I  &\cdots & \mathbf{0}& \mathbf{0}\\
    \vdots&\vdots & \vdots & \ddots & \ddots & \vdots &\vdots\\
    \mathbf{0}&\mathbf{0} &\mathbf{0} &  \cdots &\ddots &{\color{blue}{I}}&\mathbf{0}\\
    \mathbf{0}&\mathbf{0} &\mathbf{0} &\cdots &\cdots&\frac{Q}{N} & I\\
    I& \mathbf{0}&\mathbf{0} &\cdots&\cdots &\mathbf{0} &I\\
\end{bmatrix},
~~\mathbb{R}=\begin{bmatrix}
\tilde{R},\\ \mathbf{0}
\end{bmatrix}.\]
To clarify the construction of the preconditioner, we highlight one block $I$ in the matrix $\mathbb{H}$ in blue. the preconditioner $P_{\rm DE1}$ and the difference $\Delta H_1=\mathbb{H}-P_{\rm DE1}$ can be given as follows:
\begin{equation}\label{10yue30_4}
P_{\rm DE1}=\begin{bmatrix}
	I& I  & B_1(0)I & B_2(0)I &\cdots & B_{N-1}(0)I & I+B_{N}(0)I\\
	\mathbf{0}&Q &I &\mathbf{0}&\cdots&\mathbf{0}&\mathbf{0}\\
	\mathbf{0}&\mathbf{0} &\frac{Q}{2} & I &\cdots & \mathbf{0}& \mathbf{0}\\
    \vdots&\vdots & \vdots & \ddots & \ddots &\vdots &\vdots\\
    \mathbf{0}&\mathbf{0} &\mathbf{0} & \cdots &\ddots& {\color{blue}{\mathbf{0}}}&\mathbf{0}\\
    \mathbf{0}&\mathbf{0} &\mathbf{0} &\cdots&\cdots&\frac{Q}{N} &I\\
    I& \mathbf{0}&\mathbf{0} &\cdots\ &\cdots &\mathbf{0} &I\\
\end{bmatrix},~\Delta H_1=\begin{bmatrix}
\mathbf{0} &\mathbf{0} &\mathbf{0}&\cdots&\mathbf{0}&\mathbf{0}\\
\mathbf{0} &\mathbf{0} &\mathbf{0} &\cdots&\mathbf{0}&\mathbf{0}\\
\vdots & \vdots & \vdots & \vdots &\vdots &\vdots\\
\mathbf{0} & \mathbf{0} & \mathbf{0}&\cdots& {\color{blue}{I}}&\mathbf{0}\\
\mathbf{0} &\mathbf{0} &\mathbf{0}&\cdots&\mathbf{0}&\ \mathbf{0}\\
\mathbf{0}&\mathbf{0} &\mathbf{0}&\cdots &\mathbf{0} &\mathbf{0}\\
\end{bmatrix}.
\end{equation}
Fig. \ref{fig1x} depicts the sparsity patterns of matrices $\tilde{H}$, $\mathbb{H}$ and $P_{\rm DE1}$ with $N=12, M_x=M_y=8$.
\begin{figure}[ht]
\centering
\includegraphics[width=2.05in,height=1.8in]{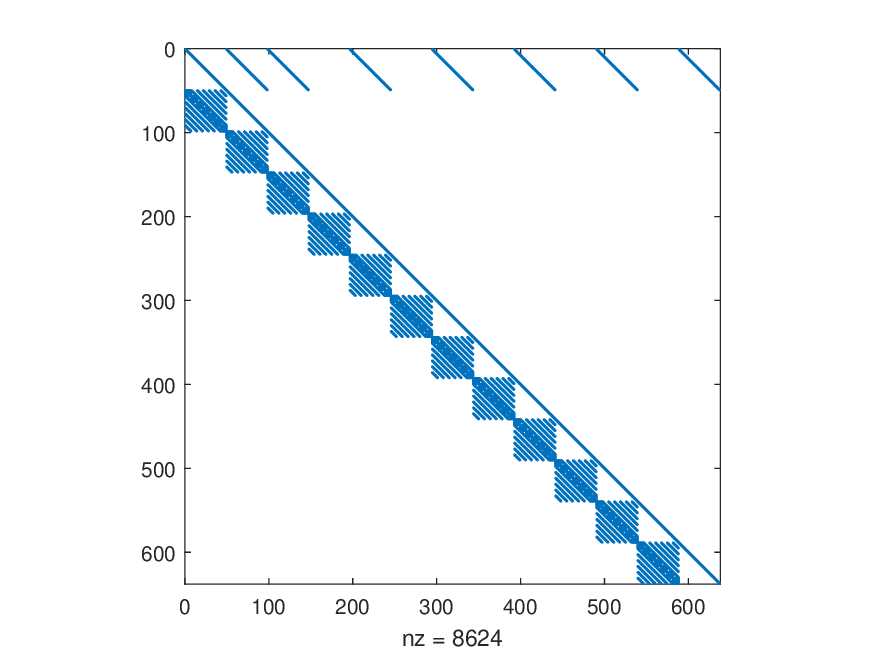}
\includegraphics[width=2.05in,height=1.8in]{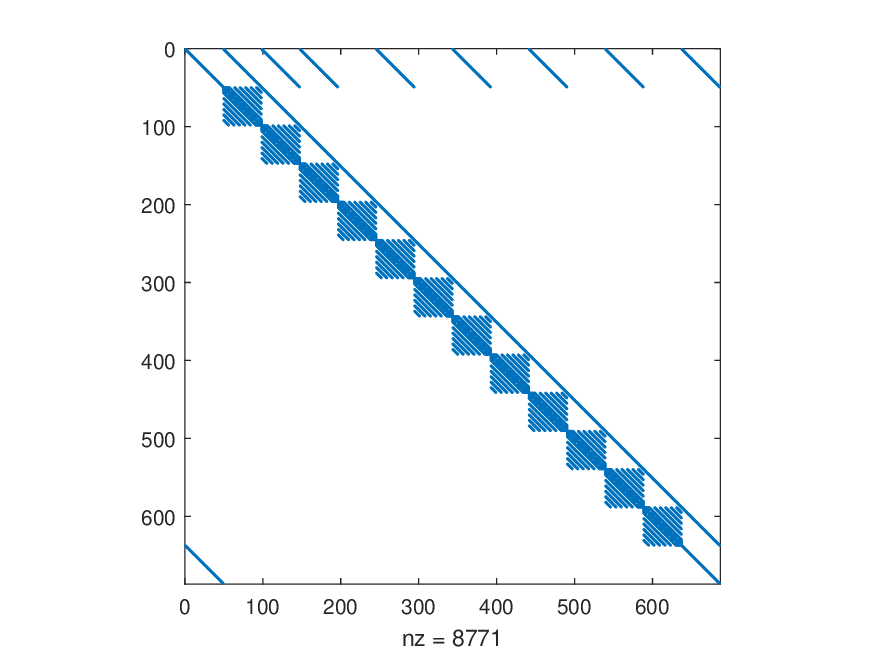}
\includegraphics[width=2.05in,height=1.8in]{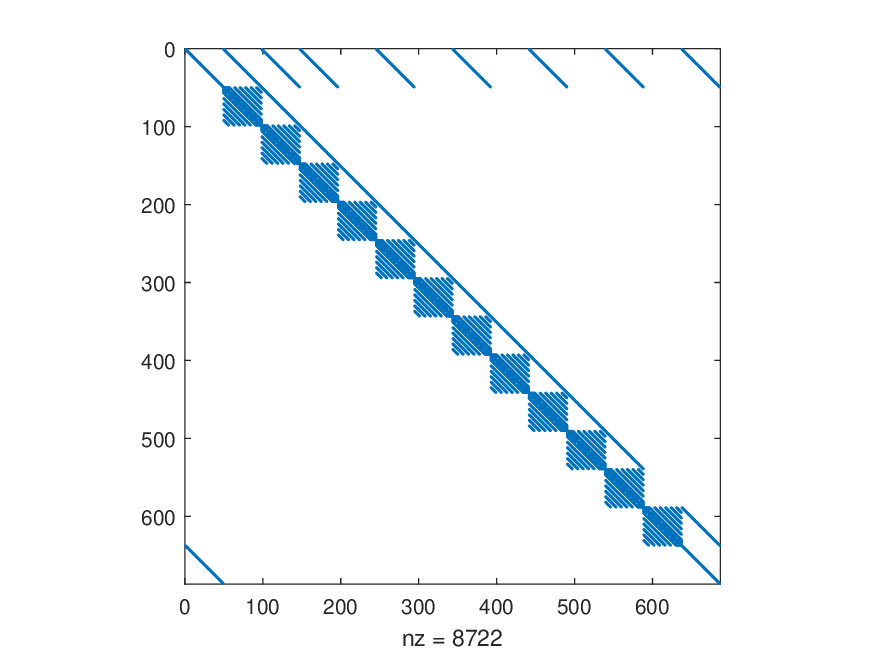}
\caption{Sparsity pattern plots of the coefficient matrix $\tilde{H}$ ($\beta_1=0,~\beta_2=1$) in (\ref{10yue30_1}) (left), the augmented matrix $\mathbb{H}$ (middle) and the preconditioner $P_{\rm DE1}$ (right).}
\label{fig1x}
\end{figure}

To further understand the eigenvalue distributions of the preconditioned matrix $P_{\rm DE1}^{-1}\mathbb{H}$, we provide the following preliminary lemma.
\begin{lemma}
	\label{lemma1}
For any nonsingular matrix $Q$, the inverse of
\[\hat{Q}=\begin{bmatrix}
Q&I\\
&\frac{Q}{2}&I\\
&&\ddots&\ddots\\
&&&\frac{Q}{N-2}&I\\
&&&&\frac{Q}{N-1}
\end{bmatrix}\]
has the form
\begin{equation}
\hat{Q}^{-1}=\begin{bmatrix}
*&*&*&s_1\\
&*&*&s_2\\
&&\ddots&\vdots\\
&&&s_{N-1}\\
\end{bmatrix},
\label{eq4.5x}
\end{equation}
where $\sigma_0=NQ^{-1}, \sigma_1=(N-1)Q^{-1}, \sigma_2=(N-2)Q^{-1}, \cdots, \sigma_{N-1}=Q^{-1}$
and
\[s_1=(-1)^{N-2}\sigma_1\sigma_2\cdots\sigma_{N-1},\
s_2=(-1)^{N-3}\sigma_1\sigma_2\cdots\sigma_{N-2},\cdots,s_{N-2}=-\sigma_1\sigma_2,\ s_{N-1}=\sigma_1,\ s_{N}=\sigma_0.\]
\end{lemma}
\begin{proof} Since
 \[\hat{Q}=\begin{bmatrix}
Q&&&\\
&\frac{Q}{2}&&\\
&&\ddots&\\
&&&\frac{Q}{N-1}\\
\end{bmatrix}
\begin{bmatrix}
I&&&\\
&\ddots&&\\
&&I&\sigma_2\\
&&&I\\
\end{bmatrix}
\begin{bmatrix}
I&&&\\
&\ddots&&\\
&&I&\sigma_3\\
&&&I\\
&&&&I\\
\end{bmatrix}\cdots
\begin{bmatrix}
I&\sigma_{N-1}&&\\
&\ddots&&\\
&&I\\
&&&I\\
\end{bmatrix},
\]
we have
 \[\hat{Q}^{-1}=
 \begin{bmatrix}
I&-\sigma_{N-1}&&\\
&\ddots&&\\
&&I\\
&&&I\\
\end{bmatrix}\cdots
\begin{bmatrix}
I&&&\\
&\ddots&&\\
&&I&-\sigma_2\\
&&&I\\
\end{bmatrix}
\begin{bmatrix}
Q^{-1}&&&\\
&2Q^{-1}&&\\
&&\ddots&\\
&&&(N-1)Q^{-1}\\
\end{bmatrix},
\]
which immediately gives Eq. (\ref{eq4.5x})
that completes the proof.  
\end{proof}

Theorem~\ref{th4} uses the matrix $\hat{Q}$ given in Lemma~\ref{lemma1} to produce a factorization that provides a practical implementation of the preconditioner $P_{\rm DE1}$.
\begin{theorem}	\label{th4}
The preconditioner $P_{\rm DE1}$ can be factorized as $P_{\rm DE1}=\Delta_1\Delta_2\Delta_3\Delta_4$, where
\[
\Delta_1=\begin{bmatrix}
	I&\mathbf{0}&\cdots&\mathbf{0}&I+B_N(0)I\\
	&   \ddots & &&\mathbf{0}\\
    &&   \ddots & &\vdots\\
    && &I &I\\
  & & &&I\\
\end{bmatrix},~\Delta_2=\begin{bmatrix}
	-B_{N}(0)I&\\
	&\hat{Q}\\
    & &\frac{Q}{N} &\\
    &  & & I\\
\end{bmatrix},\]
\[\Delta_3=\begin{bmatrix} I&\frac{-B_0(0)I}{B_N(0)}&\cdots&\frac{-B_{N-2}(0)I}{B_N(0)}&\mathbf{0}&\mathbf{0}\\
	&  I &  &&&\\
	&  &\ddots& &&\\
    &&  & I & &\\
   &   & &&I&\\
    & && & &I\\
\end{bmatrix},
~\Delta_4=\begin{bmatrix}
	I&\\
	&   \ddots&\\
    -NQ^{-1}&&   I\\
    I&& &I\\
\end{bmatrix},\]
and hence $P_{\rm DE1}$ is nonsingular if and only if $Q$ is nonsingular.
 \end{theorem}
\begin{proof} The assertion can be verified directly.  
\end{proof}

Using Lemma~\ref{lemma1} and Theorem~\ref{th4}, we can obtain the following result about the eigenvalue distributions for the preconditioned matrix $P_{\rm DE1}^{-1}\mathbb{H}$.
\begin{theorem}	\label{th5}
The preconditioned matrix  $P_{\rm DE1}^{-1}\mathbb{H}$ has an eigenvalue equal to $1$ with multiplicity at least $(N+1)\hat{k}$, and the remaining $\hat{k}$ eigenvalues are  $1+\lambda$, where $\lambda$ is the eigenvalue of
\[T=s_N\delta_1
\]
with $\delta_1=\frac{1}{B_N(0)}(B_0(0)s_1+B_1(0)s_2+\cdots+B_{N-2}(0)s_{N-1})$. Hence, when
\[||\sigma_0||=||NQ^{-1}||<\frac{-1+\sqrt{1+4\tau}}{2\tau},\] all the remaining eigenvalues $1+\lambda$ satisfy $0<|1+\lambda|<2$, where $\tau=\max_{0\leq i\leq N-2}\left|\frac{B_i(0)}{B_{N}(0)}\right|$, and $Q$ is defined in Eq.~(\ref{10yue29_1}).
\end{theorem}
\begin{proof} By Theorem \ref{th4}, we have $P_{\rm DE1}=\Delta_1\Delta_2\Delta_3\Delta_4$.
With the help of Lemma~\ref{lemma1}, we have
\[\Delta_3^{-1}\Delta_2^{-1}=\begin{bmatrix}
*&*&\cdots&\delta_1&\mathbf{0}&\mathbf{0}\\
	&  * \ & \cdots &s_1&\vdots&\vdots\\
	&  &\ \ddots\ & \vdots&\vdots&\vdots\\
    &&  & s_{N-1} &\mathbf{0} &\mathbf{0}\\
   &   & &&*&\mathbf{0}\\\
    & && & &*\\
\end{bmatrix}\quad
{\rm and}
\quad \Delta_4^{-1}\Delta_3^{-1}\Delta_2^{-1}=\begin{bmatrix}
*&*&\cdots&\delta_1&\mathbf{0}&\mathbf{0}\\
	&  * & \cdots &s_1&\vdots&\vdots\\
	&  &\ddots& \vdots&\vdots&\vdots\\
    &&  & s_{N-1} &\mathbf{0} &\mathbf{0}\\
   *&  * &\cdots &\sigma_0\delta_1&*&\mathbf{0}\\
    *&* &\cdots& -\delta_1&\mathbf{0} &*\\
\end{bmatrix}.\]
Moreover, it is evident that $\Delta_1^{-1}\Delta H_1=\Delta H_1$, from which we deduce that
\begin{equation}\label{12yue22_1}
P_{\rm DE1}^{-1}\Delta H_1
=\Delta_4^{-1}\Delta_3^{-1}\Delta_2^{-1}\Delta_1^{-1}\Delta H_1=
\begin{bmatrix}
\mathbf{0}&\mathbf{0}&\cdots&\mathbf{0}&\delta_1&\mathbf{0}\\
	\vdots&  \mathbf{0} \ &\cdots &\mathbf{0} &s_1&\vdots\\
	\vdots&  &\ \ddots\ & \vdots&\vdots\\
     \mathbf{0}& \mathbf{0}& \cdots &\mathbf{0}& s_{N-1} &\mathbf{0} \\
   \mathbf{0}&  \mathbf{0} &\cdots &\mathbf{0}&\sigma_0\delta_1&\mathbf{0}\\
    \mathbf{0}&\mathbf{0} &\cdots& \mathbf{0}&-\delta_1&\mathbf{0} \\
\end{bmatrix}.
\end{equation}
and hence, we know that $\lambda$ is the eigenvalue of $T$ noting $\sigma_0=s_N$.

Next, since $||\sigma_i||\leq ||\sigma_0||~(i=1,2,\cdots,N-1)$, we have
$||s_Ns_i||\leq ||\sigma_0||^{N-i+1}$. Consequently,
when $||\sigma_0||<\frac{-1+\sqrt{1+4\tau}}{2\tau}$, there holds
\begin{equation*}
\begin{split}
\|T\|&\leq\tau (\|s_Ns_1\|+\|s_Ns_1\|+\cdots+\|s_Ns_{N-1}\|)\\
&\leq\tau(\|\sigma_0\|^2+\|\sigma_0\|^3+\cdots+\|\sigma_0\|^N)
<\frac{\tau \|\sigma_0\|^2}{1-\|\sigma_0\|}<1,
\end{split}
\end{equation*}
which leads to $|\lambda|<1$, and $0<|1+\lambda|<2$. This completes the proof.   
\end{proof}

Fig.~\ref{fig2} depicts the eigenvalue distributions of $P_{\rm DE1}^{-1}\mathbb{H}$ and $\mathbb{H}$, where $\mathbb{H}$ is derived from Example~3 in Section~\ref{sec:numeri}.

\begin{figure}[ht]
\centering
\includegraphics[width=2.75in,height=2.0in]{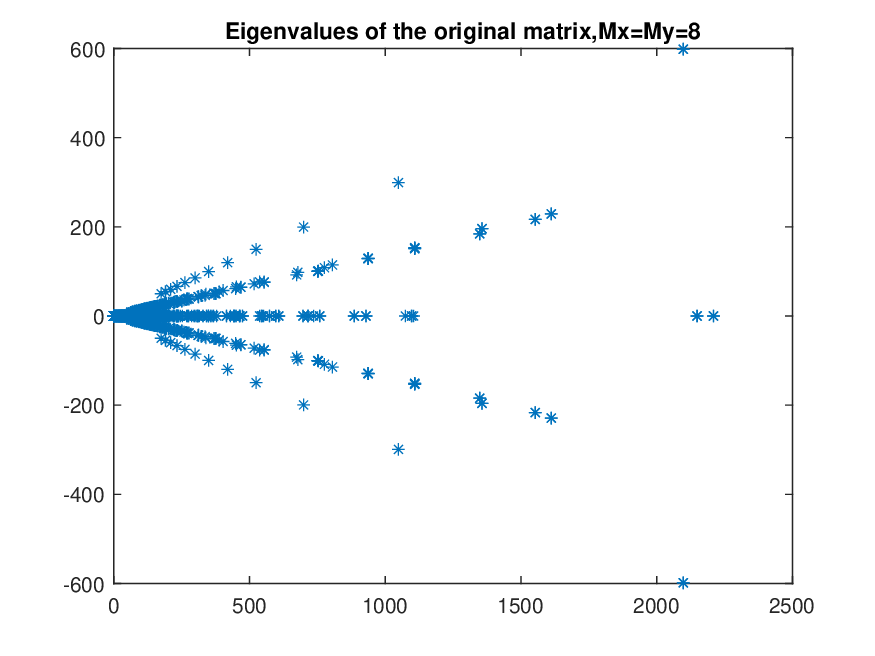}
\includegraphics[width=2.75in,height=2.0in]{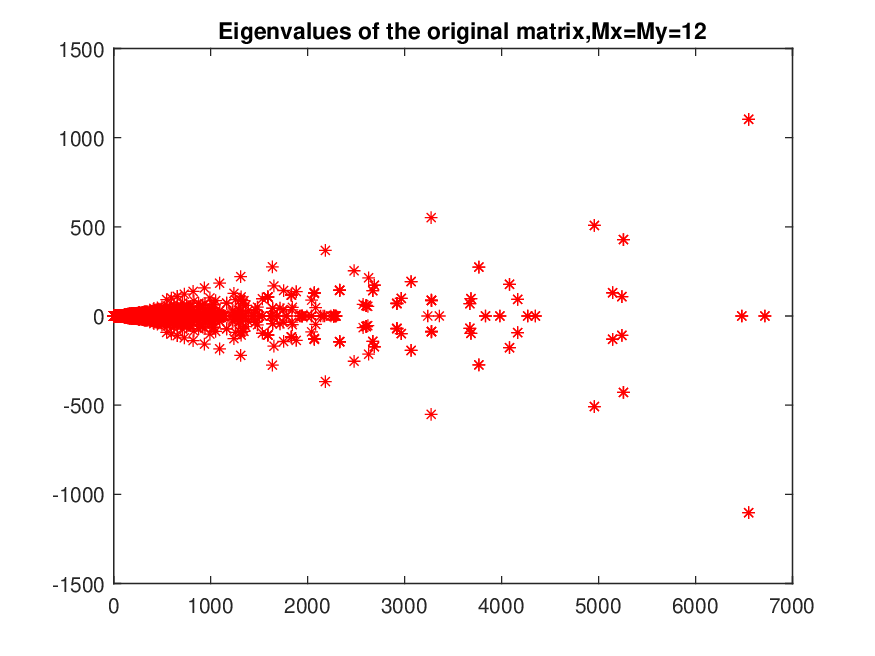}
\includegraphics[width=2.75in,height=2.0in]{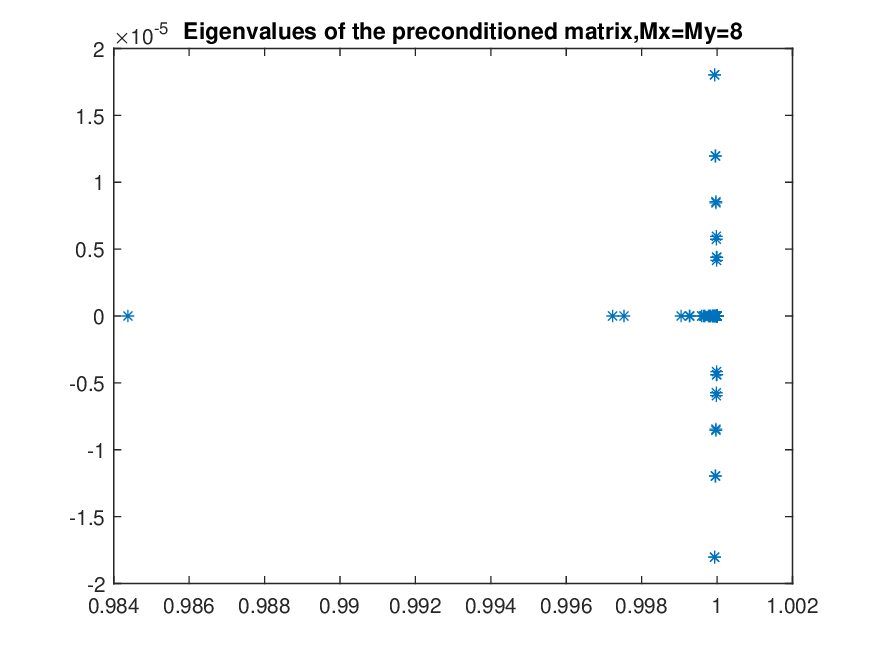}
\includegraphics[width=2.75in,height=2.0in]{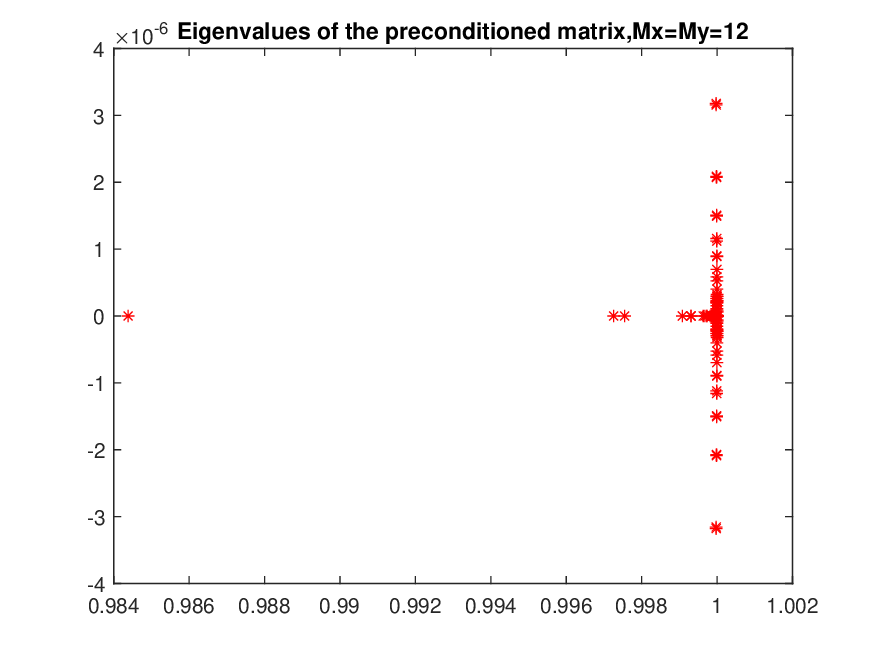}
\caption{Eigenvalue distributions of the preconditioned matrix $P_{\rm DE1}^{-1}\mathbb{H}$ and original matrix $\mathbb{H}$ in Example 3
	under $8\times 8$ and $12\times 12$ uniform grids in space, respectively.}
	\label{fig2}
\end{figure}

When $P_{\rm DE1}$ is used as a preconditioner for a Krylov subspace method like GMRES for solving Eq.~(\ref{10yue30_3}), it will provide a $(\hat{k}+1)$-dimensional search space  $\mathcal{K}(P_{\rm DE1}^{-1}\mathbb{H},$
$P_{\rm DE1}^{-1}{\bm r})$ for the solution, and hence fast convergence may be expected. The result is illustrated in Theorem~\ref{th6} below.
\begin{theorem}The degree of the minimal polynomial of the preconditioned matrix $P_{\rm DE1}^{-1}\mathbb{H}$ is at most $\hat{k}+1$. Hence, the dimension of the Krylov subspace $\mathcal{K}(P_{\rm DE1}^{-1}\mathbb{H},
	~{\bm b})$, for a given vector ${\bm b}$, is at most $\hat{k}+1$.
	\label{th6}
\end{theorem}
\begin{proof} By Eq.~(\ref{12yue22_1}),  we know that the characteristic polynomial $f(\lambda)$ of the preconditioned matrix
$P_{\rm DE1}^{-1}\mathbb{H}$ is
\[f(\lambda)=\det(\lambda I-P_{\rm DE1}^{-1}\mathbb{H})=(\lambda-1)^{(N+1)\hat{k}}\left(\lambda-\lambda_1^*\right)
\cdots\left(\lambda-\lambda_{\hat{k}}^*\right),\]
where $\lambda_i^*~(i=1,2,\cdots, \hat{k})$ are the eigenvalues of $I+\sigma_0\delta_1$. Therefore, it is clear that $f(\lambda)=0$ has at most $\hat{k}+1$ distinct roots $\lambda_i=\lambda^{*}_i,~i=1, 2, \dots, \hat{k},~\lambda_{\hat{k}+1}=1$, and this means that the degree of the minimal polynomial of $P_{\rm DE1}^{-1}\mathbb{H}$ is at most $\hat{k}+1$.
The proof of the theorem is thus completed by applying the result from~\cite{Saad}. 
\end{proof}
\subsection{Case 2: $\beta_1=1$ and $\beta_2=0$}
\label{sec:precon2}
 In this case, the coefficient matrix $H$ in Eq. (\ref{10yue5_5}) can be simplified as
 \begin{equation}\label{1yue19_2}
 H=I\otimes I_{\hat{k}}+P^2\otimes Q.
 \end{equation} 
Since the idea is very similar to that in Section~\ref{sec:precon1}, we only introduce the preconditioner for $H$ in (\ref{1yue19_2}), and all the conclusions, including the eigenvalue distributions of the preconditioned matrix, are omitted.
 
Since $P^2$ has the form
\[P^2=\begin{bmatrix}
\hat{p}_{1,1}&\hat{p}_{1,2}&\hat{p}_{1,3}&\cdots&\hat{p}_{1,N-1}&\hat{p}_{1,N}&0\\
\hat{p}_{2,1}&\hat{p}_{2,2}&\hat{p}_{2,3}&\cdots&\hat{p}_{2,N-1}&\hat{p}_{2,N}&0\\
\hat{p}_{3,1}&0&0&\cdots&0&0&0\\
0&\hat{p}_{4,2}&0&\ddots&\vdots&\vdots&\vdots\\
\vdots&\vdots&\ddots&\ddots&\ddots&\vdots&\vdots\\
\vdots&\vdots&\vdots&\ddots&\ddots&\ddots&\vdots\\
0&0&0&\cdots&\hat{p}_{N+1,N-1}&0&0\\
\end{bmatrix},\]
we know that the coefficient matrix $H$ in (\ref{1yue19_2}) has the structure
\[H=\begin{bmatrix}
I+\blacksquare_{1,1}&\blacksquare_{1,2}&\blacksquare_{1,3}&\cdots&\blacksquare_{1,N-1}
&\blacksquare_{1,N}&0\\
\blacksquare_{2,1}&I+\blacksquare_{2,2}&\blacksquare_{2,3}&\cdots&\blacksquare_{2,N-1}
&\blacksquare_{2,N}&0\\
\blacksquare_{3,1}&0&I&\cdots&0&0&0\\
0&\blacksquare_{4,2}&0&\ddots&0&0&0\\
\vdots&\vdots&\ddots&\ddots&\ddots&\vdots&\vdots\\
0&0&0&\ddots&0&I&0\\
0&0&0&\cdots&\blacksquare_{N+1,N-1}&0&I\\
\end{bmatrix},\]
where we use the symbol $\blacksquare_{i,j}=\hat{p}_{i,j}Q$ to denote different block $\hat{k}$-dimensional sub-matrices.

\begin{figure}[ht]
\centering
\includegraphics[width=2.05in,height=1.8in]{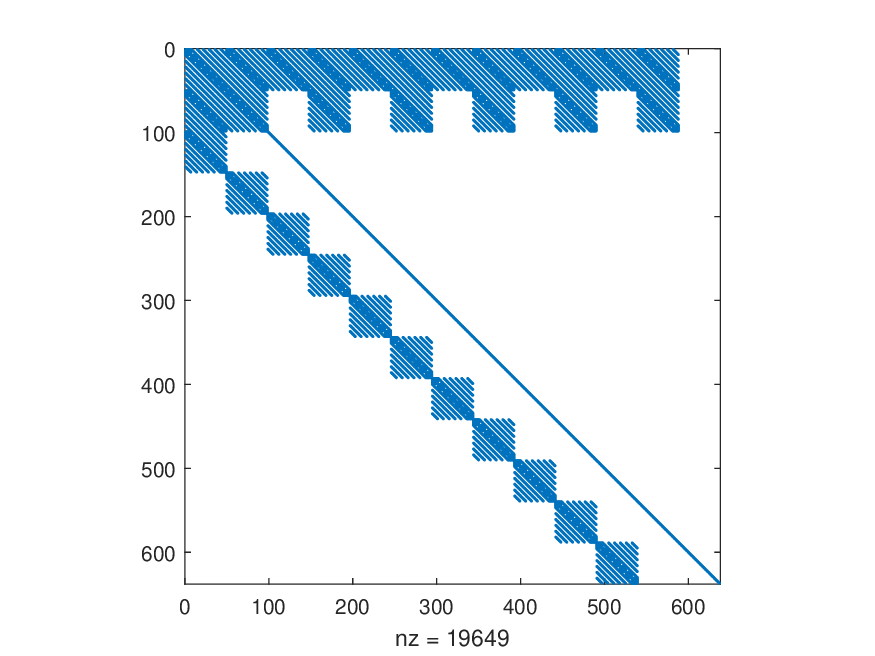}
\includegraphics[width=2.05in,height=1.8in]{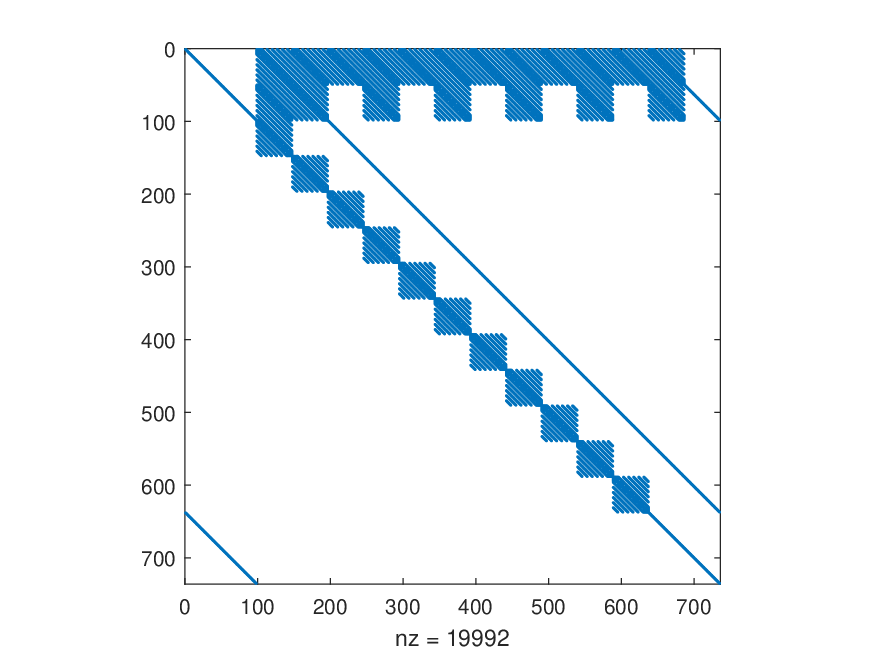}
\includegraphics[width=2.05in,height=1.8in]{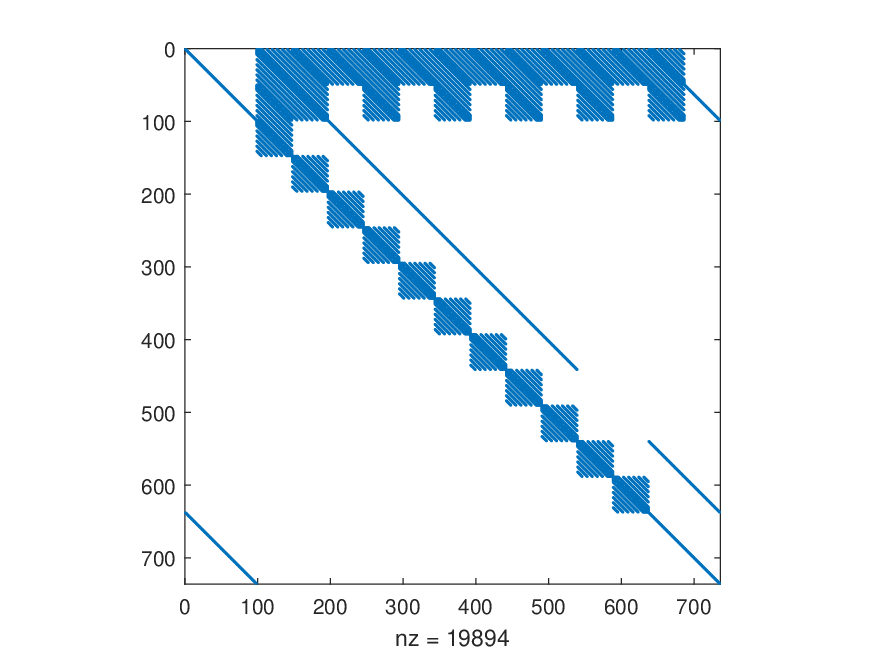}
\caption{Sparsity pattern plots of the coefficient matrix $H$ ($\beta_1=1,~\beta_2=0$) in (\ref{1yue19_2}) (left), the augmented matrix $\hat{\mathbb{H}}$ (middle) and the preconditioner $P_{\rm DE2}$ (right).}
	\label{fig3}
\end{figure}

Similar to Section \ref{sec:precon1}, we first augment (\ref{10yue5_5}) into
\begin{equation}\label{12yue29_4}
\mathbb{\hat{H}}\hat{\mathbb{U}}=\hat{\mathbb{R}},
\end{equation}
where $\hat{\mathbb{U}}=[-\mathbf{U}^{N-1},~-\mathbf{U}^{N},~\mathbf{U}^{0}, ~\mathbf{U}^{1},~\cdots,~\mathbf{U}^{N}]^\top$, and
\[\hat{\mathbb{H}}=\begin{bmatrix}
I&\mathbf{0}&I+\blacksquare_{1,1}&\blacksquare_{1,2}&\blacksquare_{1,3}&\cdots
&\blacksquare_{1,N-2}&\blacksquare_{1,N-1}&I+\blacksquare_{1,N}&\mathbf{0}\\
\mathbf{0}&I&\blacksquare_{2,1}&I+\blacksquare_{2,2}&\blacksquare_{2,3}&\cdots
&\blacksquare_{2,N-2}&\blacksquare_{2,N-1}&\blacksquare_{2,N}&I\\
\mathbf{0}&\mathbf{0}&\blacksquare_{3,1}&\mathbf{0}&I&\cdots&\mathbf{0}&
\mathbf{0}&\mathbf{0}&\mathbf{0}\\
\vdots&\vdots&\vdots&\ddots&\ddots&\ddots&\vdots&\vdots&\vdots&\vdots\\
\mathbf{0}&\mathbf{0}&\mathbf{0}&\mathbf{0}&\ddots&\ddots&{\color{blue}{I}}&
\mathbf{0}&\mathbf{0}&\mathbf{0}\\
\mathbf{0}&\mathbf{0}&\mathbf{0}&\mathbf{0}&\mathbf{0}&\ddots&\mathbf{0}
&{\color{blue}{I}}&\mathbf{0}&\mathbf{0}\\
\mathbf{0}&\mathbf{0}&\mathbf{0}&\mathbf{0}&\mathbf{0}&\ddots&\blacksquare_{N,N-2}
&\mathbf{0}&I&\mathbf{0}\\
\mathbf{0}&\mathbf{0}&\mathbf{0}&\mathbf{0}&\mathbf{0}&\cdots&0&\blacksquare_{N+1,N-1}
&\mathbf{0}&I\\
I&\mathbf{0}&\mathbf{0}&\mathbf{0}&\mathbf{0}&\ddots&\mathbf{0}&\mathbf{0}&I&\mathbf{0}\\
\mathbf{0}&I&\mathbf{0}&\mathbf{0}&\mathbf{0}&\cdots&\mathbf{0}&\mathbf{0}&\mathbf{0}&I\\
\end{bmatrix},
\hat{\mathbb{R}}=\begin{bmatrix}
R,\\ \mathbf{0} \\\mathbf{0}
\end{bmatrix}.\]
The two blue $I$ blocks in $\hat{\mathbb{H}}$ will be replaced with two blue \textbf{0} blocks in the preconditioner $P_{\rm DE2}$ given below, similar to the principle of preconditioner $P_{\rm DE1}$.
\begin{equation}\label{12yue29_5x}
P_{\rm DE2}=\begin{bmatrix}
I&0&I+\blacksquare_{1,1}&\blacksquare_{1,2}&\blacksquare_{1,3}&\cdots
&\blacksquare_{1,N-2}&\blacksquare_{1,N-1}&I+\blacksquare_{1,N}&0\\
0&I&\blacksquare_{2,1}&I+\blacksquare_{2,2}&\blacksquare_{2,3}&\cdots
&\blacksquare_{2,N-2}&\blacksquare_{2,N-1}&\blacksquare_{2,N}&I\\
0&0&\blacksquare_{3,1}&0&I&\cdots&0&0&0&0\\
\vdots&\vdots&\vdots&\ddots&\ddots&\ddots&\vdots&\vdots&\vdots&\vdots\\
0&0&0&0&\ddots&\ddots&{\color{blue}{0}}&0&0&0\\
0&0&0&0&0&\ddots&0&{\color{blue}{0}}&0&0\\
0&0&0&0&0&\ddots&\blacksquare_{N,N-2}&0&I&0\\
0&0&0&0&0&\cdots&0&\blacksquare_{N+1,N-1}&0&I\\
I&0&0&0&0&\ddots&0&0&I&0\\
0&I&0&0&0&\cdots&0&0&0&I\\
\end{bmatrix}.
\end{equation}
When $N=12, M_x=M_y=8$, the sparsity patterns of the matrices $H$, $\hat{\mathbb{H}}$ and $P_{\rm DE2}$ are depicted in Fig.~\ref{fig3}.

For a practical implementation of $P_{\rm DE2}$, we can utilize the factorization given in Theorem~\ref{th7}, which uses the following notation:
\[\theta_1=\frac{\hat{p}_{1,N-2}}{\hat{p}_{N,N-2}}, ~\theta_2=\frac{\hat{p}_{1,N-1}}{\hat{p}_{N+1,N-1}},
~\theta_3=\frac{\hat{p}_{2,N-2}}{\hat{p}_{N,N-2}},
~\theta_4= -\hat{p}_{2,N},~\theta_t = -\hat{p}_{N,N-2},~\theta_6 = -\hat{p}_{N+1,N-1},\]

\begin{theorem}	\label{th7}
The preconditioner $P_{\rm DE2}$ can be factorized as $P_{\rm DE2}=\hat{\Delta}_1\hat{\Delta}_2\hat{\Delta}_3\hat{\Delta}_4$, where
\[
\hat{\Delta}_1=
\begin{bmatrix}
\theta_2I&\theta_4Q+\theta_{1}I&I+\blacksquare_{1,1}&\blacksquare_{1,2}
&\blacksquare_{1,3}&\cdots
&\blacksquare_{1,N-2}&\blacksquare_{1,N-1}&I+\blacksquare_{1,N}&0\\
0&\theta_4Q+\theta_3I&\blacksquare_{2,1}&I+\blacksquare_{2,2}&\blacksquare_{2,3}&\cdots
&\blacksquare_{2,N-2}&\blacksquare_{2,N-1}&\blacksquare_{2,N}&I\\
0&0&\blacksquare_{3,1}&0&I&\cdots&0&0&0&0\\
\vdots&\vdots&\vdots&\ddots&\ddots&\ddots&\ddots&\vdots&\vdots&\vdots\\
0&0&0&0&\ddots&\ddots&{\color{blue}{0}}&0&0&0\\
0&0&0&0&0&\ddots&0&{\color{blue}{0}}&0&0\\
0&0&0&0&0&\ddots&\blacksquare_{N,N-2}&0&I&0\\
0&0&0&0&0&\cdots&0&\blacksquare_{N+1,N-1}&0&I\\
0&0&0&0&0&\ddots&0&0&I&0\\
0&0&0&0&0&\cdots&0&0&0&I\\
\end{bmatrix},\]
\[\hat{\Delta}_2=\begin{bmatrix}
	0&I&\\
	I&0&\\
	& &I\\
	&&&\ddots &\\
	&  & & &I\\
\end{bmatrix},~\hat{\Delta}_3=\begin{bmatrix}
I&\\
	&  I \\
	&  &\ddots\\
\frac{-Q^{-1}}{\theta_5} &&  & I\\
   & \frac{-Q^{-1}}{\theta_6}  & &&I&\\
    & && & &&I\\
    & && & &&&I\\
\end{bmatrix},~
\hat{\Delta}_4=\begin{bmatrix}
	I&\\
    &I\\
	&   &\ddots &\\
    I&& &  I\\
    &I& &&I\\
\end{bmatrix}.\]
Hence, $P_{\rm DE2}$ is nonsingular if and only if $Q$ and $ \theta_4Q+\theta_3I$ are nonsingular.
 \end{theorem}
\begin{proof}
The assertion can be verified directly. 
\end{proof}

Fig.~\ref{fig4} depicts the eigenvalue distributions of $P_{\rm DE2}^{-1}\hat{\mathbb{H}}$ and $\hat{\mathbb{H}}$, where $\hat{\mathbb{H}}$ comes from Example~4 in Section~\ref{sec:numeri}.
\begin{figure}[ht]
\centering
\includegraphics[width=2.8in,height=2.0in]{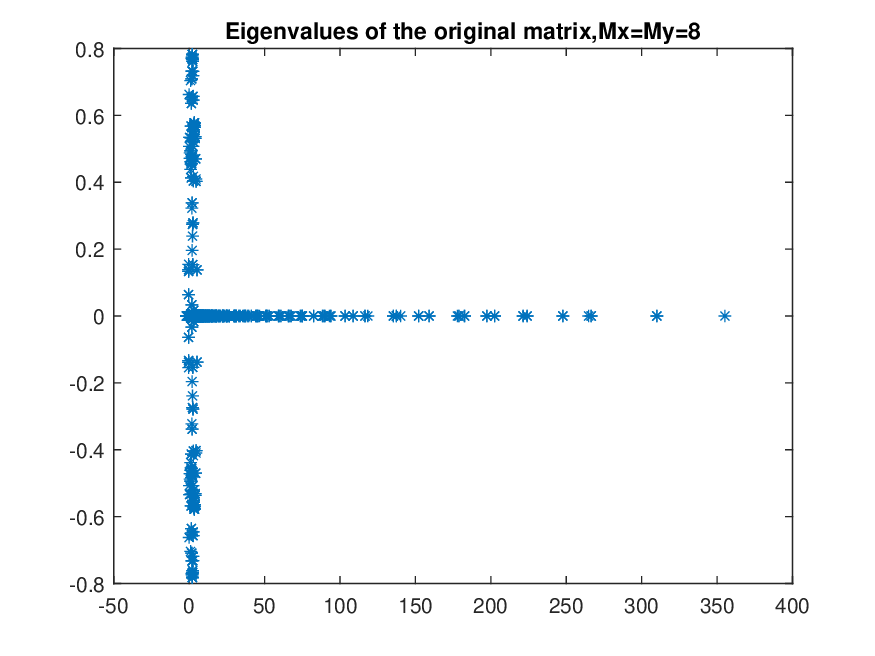}
\includegraphics[width=2.8in,height=2.0in]{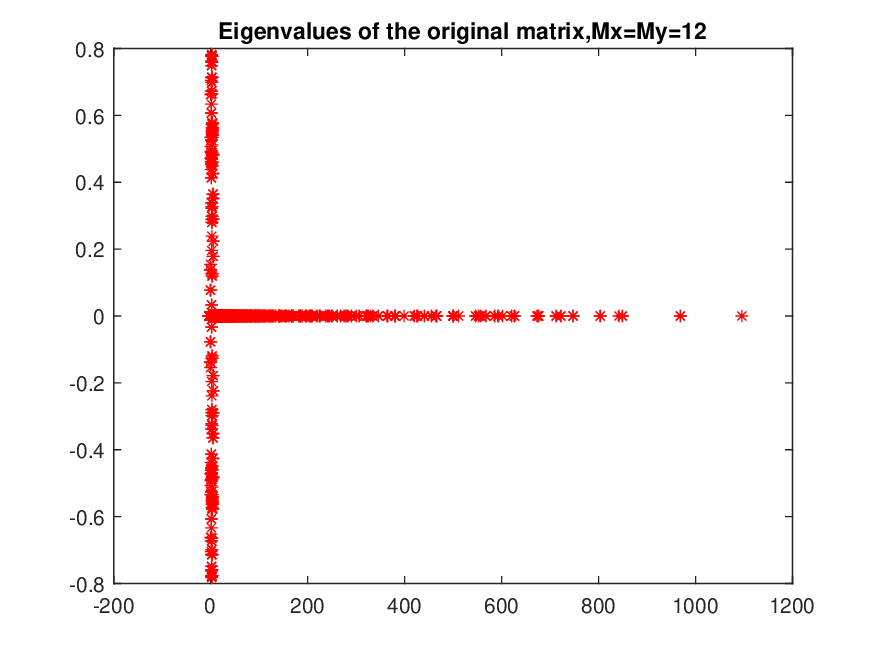}
\includegraphics[width=2.8in,height=2.0in]{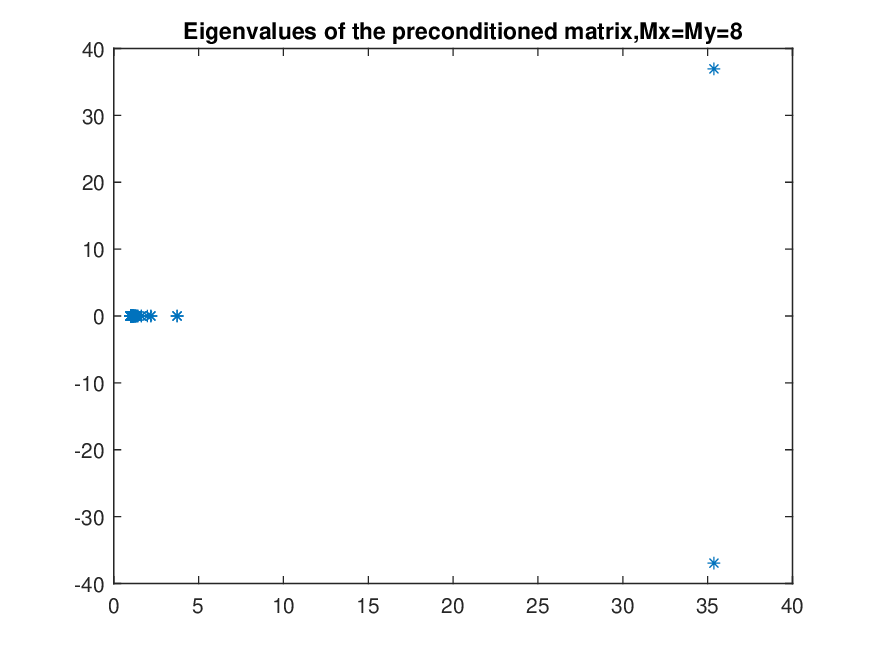}
\includegraphics[width=2.8in,height=2.0in]{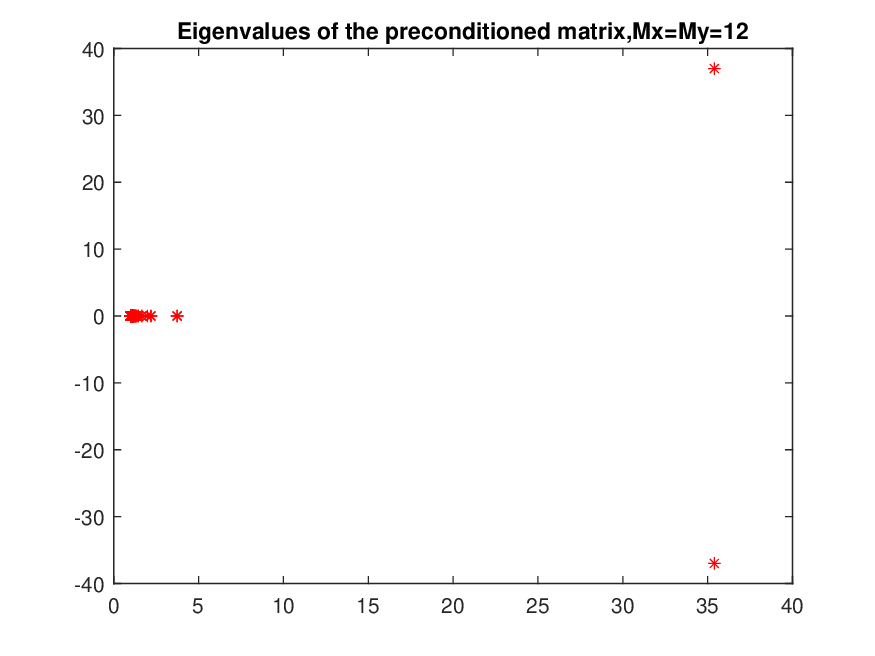}
\caption{Eigenvalue distributions of the preconditioned matrix $P_{\rm DE2}^{-1}\hat{\mathbb{H}}$ and original matrix $\hat{\mathbb{H}}$
in Example 4 under $8\times 8$ and $12\times 12$ uniform grids in space, respectively.}
\label{fig4}
\end{figure}
\subsection{Case 3: $\beta_1\neq 0$ and $\beta_2\neq 0$}
\label{sec:precon3x}
In this case, the coefficient matrix $H$ in~Eq. (\ref{10yue5_5}) enjoys the following structure
\begin{equation}\label{1yue19_3}
H=
\begin{bmatrix}
\blacklozenge_{1,1}&\blacklozenge_{1,2}&\blacklozenge_{1,3}&\cdots&\blacklozenge_{1,N-1}
&\blacklozenge_{1,N}&0\\
\beta_2I+\blacksquare_{2,1}&\beta_1I+\blacksquare_{2,2}&\blacksquare_{2,3}
&\cdots&\blacksquare_{2,N-1}&\blacksquare_{2,N}&0\\
\blacksquare_{3,1}&\frac{\beta_2I}{2}&\beta_1I&\cdots&0&0&0\\
0&\blacksquare_{4,2}&\frac{\beta_2I}{3}&\ddots&0&0&0\\
\vdots&\vdots&\ddots&\ddots&\ddots&\vdots&\vdots\\
0&0&0&\ddots&\frac{\beta_2I}{N-1}&\beta_1I&0\\
0&0&0&\cdots&\blacksquare_{N+1,N-1}&\frac{\beta_2I}{N}&\beta_1I\\
\end{bmatrix},
\end{equation}
where $\blacklozenge_{1,1}=v_{1}I+\blacksquare_{1,1}, ~\blacklozenge_{1,2}=v_{2}I+\blacksquare_{1,2},
~\cdots,~\blacklozenge_{1,N}=v_{N}I+\blacksquare_{1,N}$, with
$v_1=\beta_1-\beta_2B_1(0),~v_2=-\frac{\beta_2B_2(0)}{2},~\cdots,
~v_N=-\frac{\beta_2B_N(0)}{N}$.

\begin{figure}[ht]
\centering
\includegraphics[width=2.05in,height=1.8in]{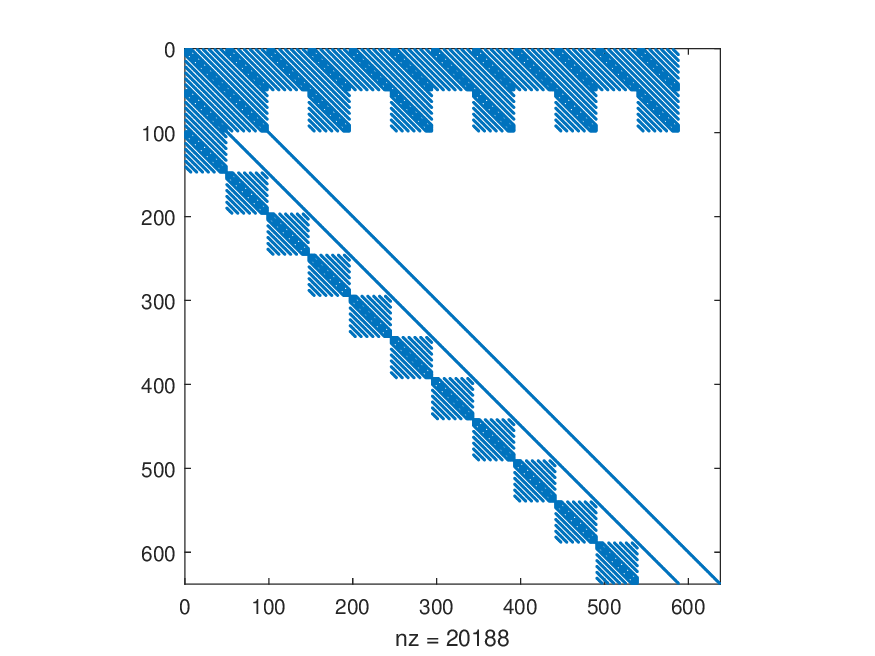}
\includegraphics[width=2.05in,height=1.8in]{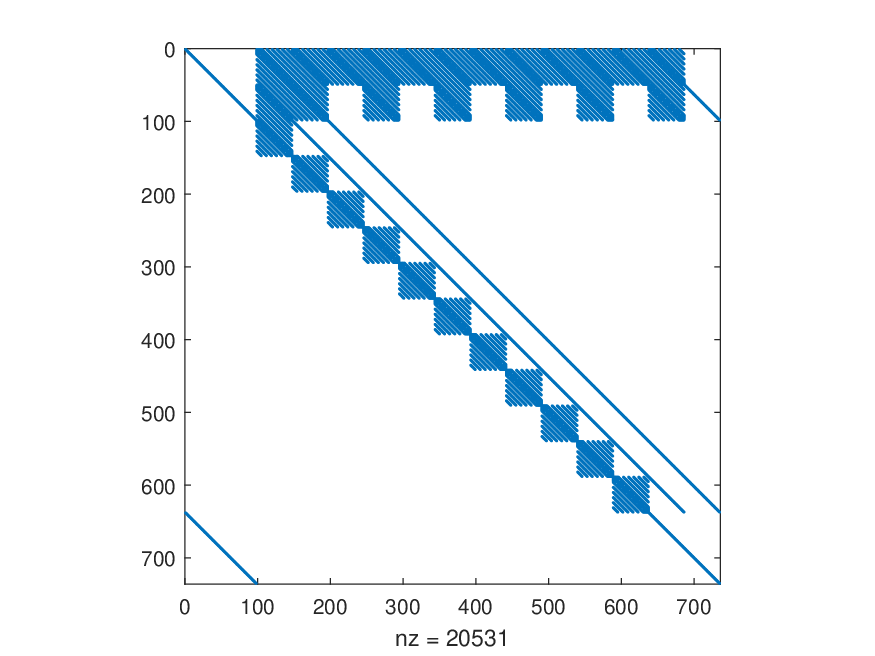}
\includegraphics[width=2.05in,height=1.8in]{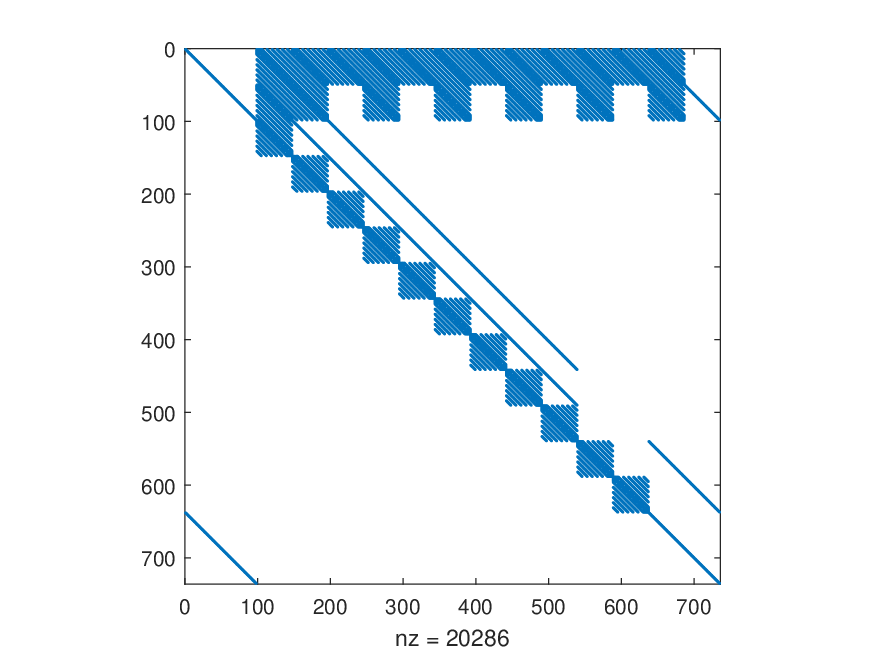}
\caption{Sparsity pattern plots of the coefficient matrix $H$ ($\beta_1\neq0,~\beta_2\neq0$) in (\ref{1yue19_3}) (left), the augmented matrix $\tilde{\mathbb{H}}$ (middle) and the preconditioner $P_{\rm DE3}$ (right).}
	\label{fig5}
\end{figure}

Similar to Section \ref{sec:precon1}, we first augment (\ref{10yue5_5}) into
\begin{equation}\label{1yue20_1}
\mathbb{\tilde{H}}\hat{\mathbb{U}}=\hat{\mathbb{R}},
\end{equation}
where
\[\tilde{\mathbb{H}}=\begin{bmatrix}
I&\mathbf{0}&\blacklozenge_{1,1}&\blacklozenge_{1,2}&\blacklozenge_{1,3}
&\cdots&\blacklozenge_{1,N-2}&\blacklozenge_{1,N-1}&I+\blacklozenge_{1,N}&\mathbf{0}\\
\mathbf{0}&I&\beta_2I+\blacksquare_{2,1}&\beta_1I+\blacksquare_{2,2}&\blacksquare_{2,3}
&\cdots&\blacksquare_{2,N-2}&\blacksquare_{2,N-1}&\blacksquare_{2,N}&I\\
\mathbf{0}&\mathbf{0}&\blacksquare_{3,1}&\frac{\beta_2I}{2}&\beta_1I&\cdots&\mathbf{0}&
\mathbf{0}&\mathbf{0}&\mathbf{0}\\
\vdots&\vdots&\vdots&\ddots&\ddots&\ddots&\vdots&\vdots&\vdots&\vdots\\
\mathbf{0}&\mathbf{0}&\mathbf{0}&\mathbf{0}&\ddots&\ddots&{\color{blue}{\beta_1I}}&
\mathbf{0}&\mathbf{0}&\mathbf{0}\\
\mathbf{0}&\mathbf{0}&\mathbf{0}&\mathbf{0}&\mathbf{0}&\ddots
&{\color{blue}{\frac{\beta_2I}{N-2}}}
&{\color{blue}{\beta_1I}}&\mathbf{0}&\mathbf{0}\\
\mathbf{0}&\mathbf{0}&\mathbf{0}&\mathbf{0}&\mathbf{0}&\ddots&\blacksquare_{N,N-2}
&{\color{blue}{\frac{\beta_2I}{N-1}}}&\beta_1 I&\mathbf{0}\\
\mathbf{0}&\mathbf{0}&\mathbf{0}&\mathbf{0}&\mathbf{0}&\cdots&0&\blacksquare_{N+1,N-1}
&{\color{blue}{\frac{\beta_2I}{N}}}&\beta_1I\\
I&\mathbf{0}&\mathbf{0}&\mathbf{0}&\mathbf{0}&\ddots&\mathbf{0}&\mathbf{0}&I&\mathbf{0}\\
\mathbf{0}&I&\mathbf{0}&\mathbf{0}&\mathbf{0}&\cdots&\mathbf{0}&\mathbf{0}&\mathbf{0}&I\\
\end{bmatrix}.\]
Here we should note that the five blue blocks in $\tilde{\mathbb{H}}$ will be replaced with five blue \textbf{0} blocks in the $P_{\rm DE3}$ preconditioner given below:
\begin{equation}\label{12yue29_5}
P_{\rm DE3}=
\begin{bmatrix}
I&\mathbf{0}&\blacklozenge_{1,1}&\blacklozenge_{1,2}&\blacklozenge_{1,3}
&\cdots&\blacklozenge_{1,N-2}&\blacklozenge_{1,N-1}&I+\blacklozenge_{1,N}&\mathbf{0}\\
\mathbf{0}&I&\beta_2I+\blacksquare_{2,1}&\beta_1I+\blacksquare_{2,2}&\blacksquare_{2,3}
&\cdots&\blacksquare_{2,N-2}&\blacksquare_{2,N-1}&\blacksquare_{2,N}&I\\
\mathbf{0}&\mathbf{0}&\blacksquare_{3,1}&\frac{\beta_2I}{2}&\beta_1I&\cdots&\mathbf{0}&
\mathbf{0}&\mathbf{0}&\mathbf{0}\\
\vdots&\vdots&\vdots&\ddots&\ddots&\ddots&\vdots&\vdots&\vdots&\vdots\\
\mathbf{0}&\mathbf{0}&\mathbf{0}&\mathbf{0}&\ddots&\ddots&{\color{blue}{\mathbf{0}}}&
\mathbf{0}&\mathbf{0}&\mathbf{0}\\
\mathbf{0}&\mathbf{0}&\mathbf{0}&\mathbf{0}&\mathbf{0}&\ddots
&{\color{blue}{\mathbf{0}}}
&{\color{blue}{\mathbf{0}}}&\mathbf{0}&\mathbf{0}\\
\mathbf{0}&\mathbf{0}&\mathbf{0}&\mathbf{0}&\mathbf{0}&\ddots&\blacksquare_{N,N-2}
&{\color{blue}{\mathbf{0}}}&\beta_1 I&\mathbf{0}\\
\mathbf{0}&\mathbf{0}&\mathbf{0}&\mathbf{0}&\mathbf{0}&\cdots&0&\blacksquare_{N+1,N-1}
&{\color{blue}{\mathbf{0}}}&\beta_1I\\
I&\mathbf{0}&\mathbf{0}&\mathbf{0}&\mathbf{0}&\ddots&\mathbf{0}&\mathbf{0}&I&\mathbf{0}\\
\mathbf{0}&I&\mathbf{0}&\mathbf{0}&\mathbf{0}&\cdots&\mathbf{0}&\mathbf{0}&\mathbf{0}&I\\
\end{bmatrix}.
\end{equation}
Fig.~\ref{fig5} depicts the sparsity patterns of the matrices $H$, $\tilde{\mathbb{H}}$ and $P_{\rm DE3}$ under $N=12, M_x=M_y=8$. Similarly, $P_{\rm DE3}$ can be represented in factorized form according to Theorem~\ref{th8} below.

\begin{theorem}	\label{th8}
The preconditioner $P_{\rm DE3}$ can be factorized as $P_{\rm DE3}=\tilde{\Delta}_1\tilde{\Delta}_2\tilde{\Delta}_3\tilde{\Delta}_4\tilde{\Delta}_5$, where
\[
\tilde{\Delta}_1=
\begin{bmatrix}
	\tilde{\Theta}_1&\tilde{\Theta}_2&\blacklozenge_{1,1}&\blacklozenge_{1,2}&\blacklozenge_{1,3}
	&\cdots&\blacklozenge_{1,N-2}&\blacklozenge_{1,N-1}&I+\blacklozenge_{1,N}&\mathbf{0}\\
	\mathbf{0}&\tilde{\Theta}_3&\beta_2I+\blacksquare_{2,1}&\beta_1I+\blacksquare_{2,2}&\blacksquare_{2,3}
	&\cdots&\blacksquare_{2,N-2}&\blacksquare_{2,N-1}&\blacksquare_{2,N}&I\\
	\mathbf{0}&\mathbf{0}&\blacksquare_{3,1}&\frac{\beta_2I}{2}&\beta_1I&\cdots&\mathbf{0}&
	\mathbf{0}&\mathbf{0}&\mathbf{0}\\
	\vdots&\vdots&\vdots&\ddots&\ddots&\ddots&\vdots&\vdots&\vdots&\vdots\\
	\mathbf{0}&\mathbf{0}&\mathbf{0}&\mathbf{0}&\ddots&\ddots&{\color{blue}{\mathbf{0}}}&
	\mathbf{0}&\mathbf{0}&\mathbf{0}\\
	\mathbf{0}&\mathbf{0}&\mathbf{0}&\mathbf{0}&\mathbf{0}&\ddots
	&{\color{blue}{\mathbf{0}}}
	&{\color{blue}{\mathbf{0}}}&\mathbf{0}&\mathbf{0}\\
	\mathbf{0}&\mathbf{0}&\mathbf{0}&\mathbf{0}&\mathbf{0}&\ddots&\blacksquare_{N,N-2}
	&{\color{blue}{\mathbf{0}}}&\beta_1 I&\mathbf{0}\\
	\mathbf{0}&\mathbf{0}&\mathbf{0}&\mathbf{0}&\mathbf{0}&\cdots&0&\blacksquare_{N+1,N-1}
	&{\color{blue}{\mathbf{0}}}&\beta_1I\\
	\mathbf{0}&\mathbf{0}&\mathbf{0}&\mathbf{0}&\mathbf{0}&\ddots&\mathbf{0}&\mathbf{0}&I&\mathbf{0}\\
	\mathbf{0}&\mathbf{0}&\mathbf{0}&\mathbf{0}&\mathbf{0}&\cdots&\mathbf{0}&\mathbf{0}&\mathbf{0}&I\\
\end{bmatrix},
\]
\[\tilde{\Delta}_2=\begin{bmatrix}
	\blacksquare^{-1}_{N+1,N-1}&&\\
	&I&\\
    && \ddots &\\
    &  & & I\\
\end{bmatrix},
~~\tilde{\Delta}_3=\begin{bmatrix}
	0&I&\\
	I&0&\\
    & &I\\
    &&&\ddots &\\
    &  & & &I\\
\end{bmatrix},\]
\[\tilde{\Delta}_4=\begin{bmatrix}
I&\\
	&  I \\
	&  &\ddots\\
-\beta_1\blacksquare^{-1}_{N,N-2} &&  & I\\
   & -\beta_1\blacksquare^{-1}_{N+1,N-1} & &&I&\\
    & && & &&I\\
    & && & &&&I\\
\end{bmatrix}
~{\rm and}~\tilde{\Delta}_5=\begin{bmatrix}
	I&\\
    &I\\
	&   &\ddots&\\
    I&& &  I\\
    &I& &&I\\
\end{bmatrix}\]
with $\tilde{\Theta}_1=\beta_1(\blacksquare_{1,N-1}+v_{N-1}I)$,
$\tilde{\Theta}_2=-\blacksquare_{1,N}
+v_{N-2}\beta_1\blacksquare^{-1}_{N,N-2}
+(\frac{\beta_1\hat{p}_{1,N-2}}{\hat{p}_{N,N-2}}-v_N)I$, $\tilde{\Theta}_3=-\blacksquare_{2,N}+\frac{\beta_1\hat{p}_{2,N-2}}{\hat{p}_{N,N-2}}I$,
and hence, $P_{\rm DE3}$ is nonsingular if and only if $Q, \tilde{\Theta}_1, \tilde{\Theta}_3$ are nonsingular.
 \end{theorem}
\begin{proof} The assertion can be verified directly.  
\end{proof}

Fig.~\ref{fig6} depicts the eigenvalue distributions of $P_{\rm DE3}^{-1}\tilde{\mathbb{H}}$ and  $\tilde{\mathbb{H}}$, where $\tilde{\mathbb{H}}$ is derived from Example~5 in the next section.
\begin{figure}[ht]
\centering
\includegraphics[width=2.8in,height=2.0in]{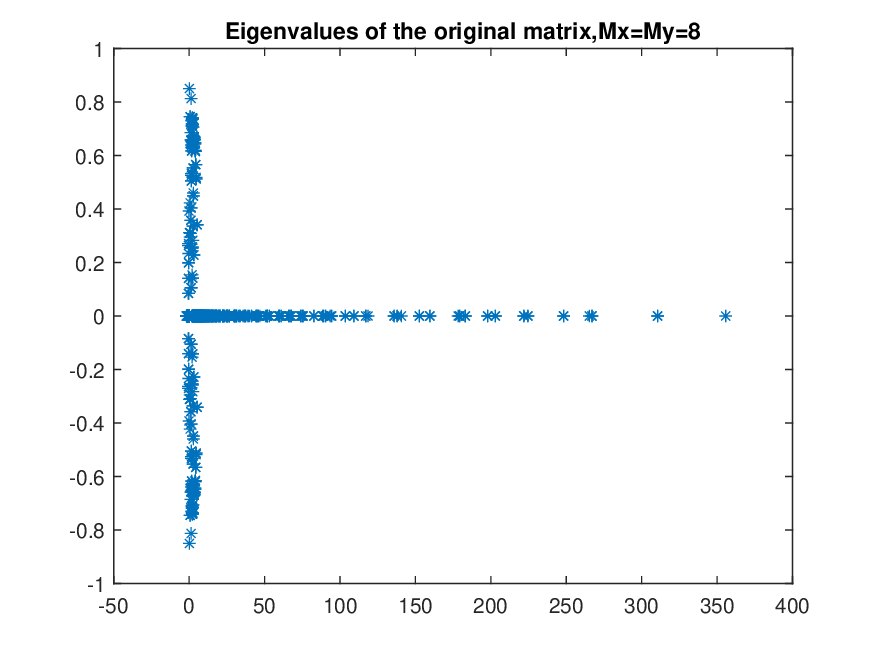}
\includegraphics[width=2.8in,height=2.0in]{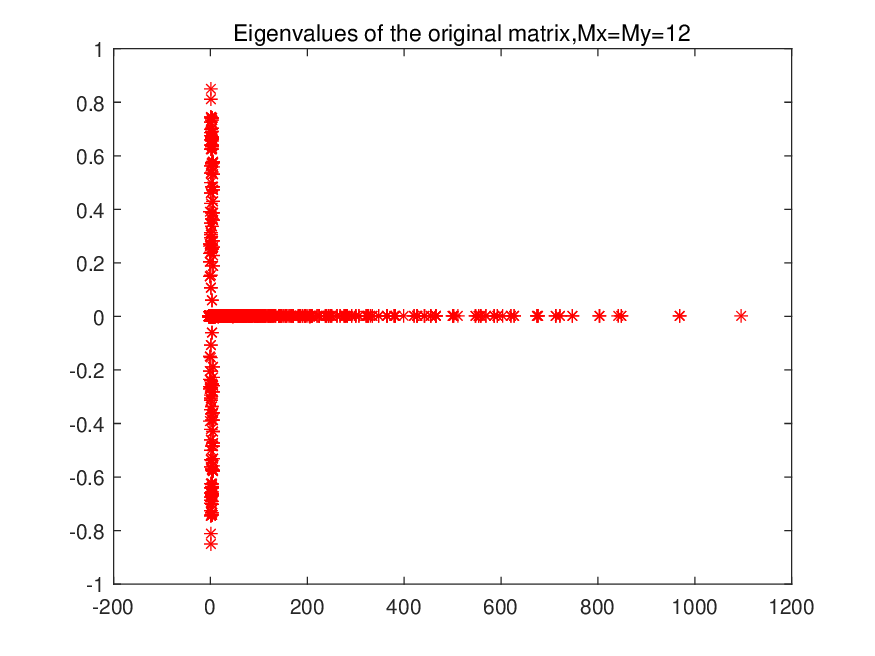}
\includegraphics[width=2.8in,height=2.0in]{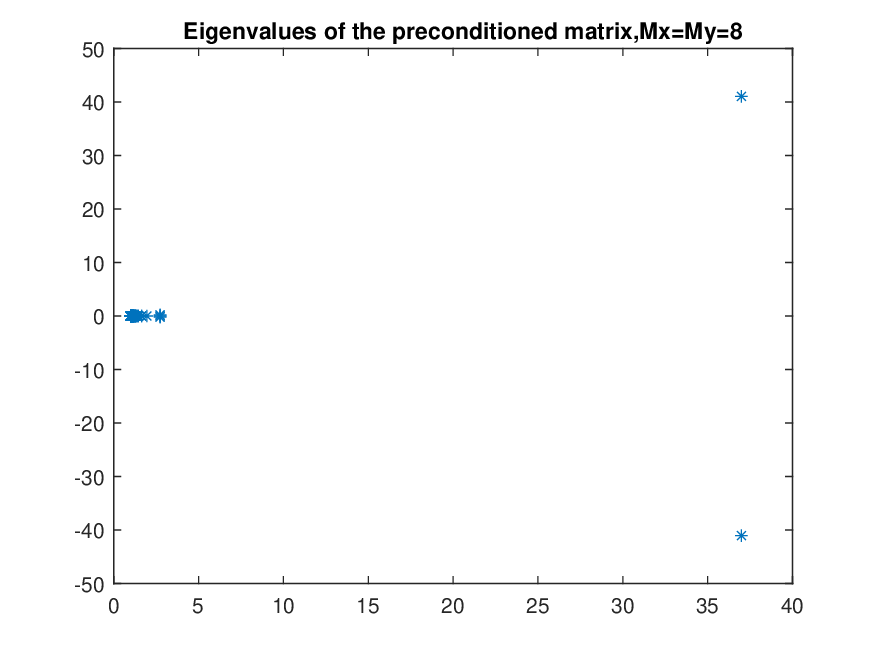}
\includegraphics[width=2.8in,height=2.0in]{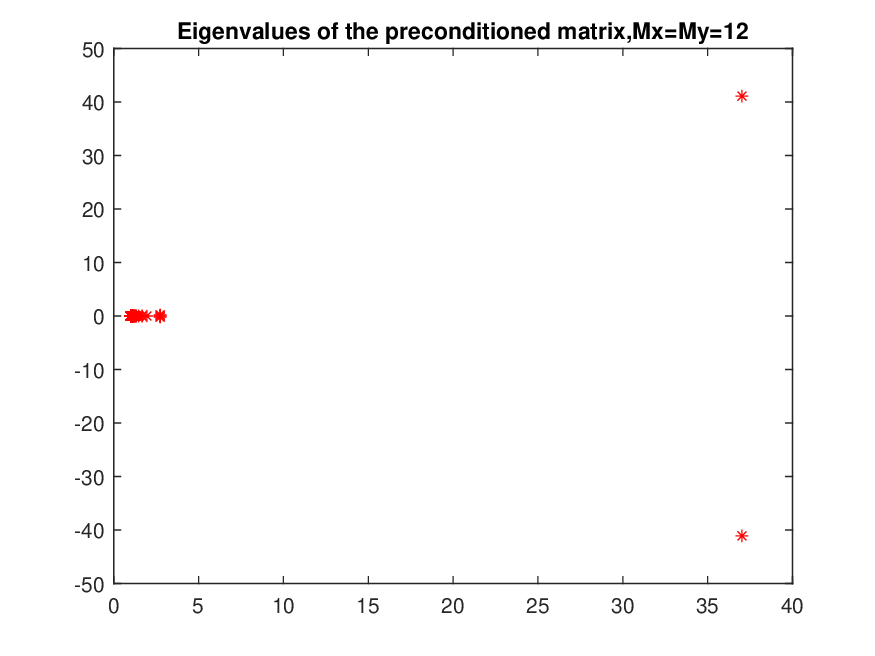}
\caption{Eigenvalue distributions of the preconditioned matrix $P_{\rm DE3}^{-1}\tilde{\mathbb{H}}$ and original matrix $\tilde{\mathbb{H}}$
in Example 5 under $8\times 8$ and $12\times 12$ uniform grids in space, respectively.}
	\label{fig6}
\end{figure}
\begin{remark}\label{rem4.7x}
When $P_{\rm DE1}$, $P_{\rm DE2}$ and $P_{\rm DE3}$ are used as preconditioners for a Krylov subspace method, the computational cost of their construction and application can be estimated using Theorems \ref{th4}, \ref{th7} and \ref{th8}, as described below:


\begin{enumerate}
\item Implementing $P_{\rm DE1}$ requires to compute the LU factorization of the submatrix $Q$ in (\ref{10yue5_5}). Applying $P_{\rm DE1}$ at each Krylov iteration requires to solve $N+1$ subsystems $Qv=r$ per iteration utilizing the computed $L$, $U$ factors. The overall cost in terms of number of arithmetic operations is  $\mathcal{O}\left(\hat{k}^3+(N+1)\hat{k}^2\right)$;
\item Implementing $P_{\rm DE2}$ requires to compute the LU factorizations of the submatrix $Q$ in~(\ref{10yue5_5}) and of matrix  $\theta_4Q+\theta_3I$ in Theorem~\ref{th7}. Applying $P_{\rm DE2}$ at each Krylov iteration requires to solve $N+1$ subsystems $Qv=r$ and one subsystem $(\theta_4Q+\theta_3I)v=r$ using the computed $L,~U$ factors. The overall cost is $\mathcal{O}\left(2\hat{k}^3+(N+2)\hat{k}^2\right)$;
\item Finally, implementing $P_{\rm DE3}$ requires to compute the LU factorizations of the submatrix $Q$ in~(\ref{10yue5_5}) and of matrices $\tilde{\Theta}_1,~\tilde{\Theta}_3$ in Theorem~\ref{th8}. Applying $P_{\rm DE3}$ at each Krylov iteration requires to solve $N+2$ subsystems $Qv=r$ and two subsystems $\tilde{\Theta}_1v=r,~\tilde{\Theta}_3v=r$ using the computed $L$, $U$ factors. The overall cost is $\mathcal{O}\left(3\hat{k}^3+(N+4)\hat{k}^2\right)$ operations.
\end{enumerate} 
In contrast, the direct solution $U=H\backslash R$ of system~(\ref{10yue5_5}) has a significantly higher $\mathcal{O}\left((N+1)^3\hat{k}^3\right)$
computational cost compared to the implementation of $P_{\rm DE1}$, $P_{\rm DE2}$, or $P_{\rm DE3}$.
\end{remark}
\section{Numerical results}
\label{sec:numeri}
In this section, we evaluate the performance of the presented {\em B-brm-c} method on some numerical examples using MATLAB R2016(a) on a PC powered by an Intel(R) Core(TM) i5-8265U processor (CPU@1.60GHz). We compare this performance to that of the conventional barycentric rational collocation method (referred to as {\em cbrc}), which approximates the unknown solution function by using $\hat{u}(t,x,y)=\sum^{\hat{M},\tilde{M},\overline{M}}_{k=0,i=0,j=0}
\varphi_k(t)\varphi_i(x)\varphi_j(y)u_{k,i,j}$ in Eq.~(\ref{9yue26_2}).
This {\em cbrc} method has been extensively applied to the numerical solution of numerous PDEs,  for example, see \cite{luo1,liufei1,lijin1}. The coefficient matrix of the relevant linear system for the {\em cbrc} method cannot be readily transformed into one with block-structured form, which limits the application of fast matrix solvers.
Because the theoretical accuracy in the $x$-direction is the same as in the $y$-direction, we always use $h_x=h_y,~\tilde{d}=5$ when utilizing the {\em B-brm-c} and {\em cbrc} methods.

In our experiments we solve the linear systems $\mathbb{H}\mathbb{U}=\mathbb{R}$ in Eq.~(\ref{10yue30_3}), $\mathbb{\hat{H}}\hat{\mathbb{U}}=\hat{\mathbb{R}}$ in Eq.~(\ref{12yue29_4}) and $\mathbb{\tilde{H}}\hat{\mathbb{U}}=\hat{\mathbb{R}}$ in Eq.~(\ref{1yue20_1}) iteratively using the GMRES method~\cite{Saad}
which is preconditioned from the right by the
$P_{\rm DE1}$, $P_{\rm DE2}$  and $P_{\rm DE3}$ preconditioners (described in Section~\ref{sec:precon}). In GMRES, the stopping tolerance is ${\tt tol} = 10^{-10}$, and the dimension of the Krylov subspace is set to ${\tt restart} = 30$. Before executing the GMRES method, we compute the LU factorization of submatrix $Q$ in~(\ref{10yue30_3}), ~(\ref{12yue29_4}) and (\ref{1yue20_1}), and $\theta_4Q+\theta_3I$ in Theorem \ref{th7}, and $\tilde{\Theta}_1, \tilde{\Theta}_3$ in Theorem \ref{th8}. The overall solution time reported in our experiments for solving the linear system includes the CPU elapsed time required to complete these matrix factorizations. Given the dense structure of the coefficient matrix $\mathbf{B}$ in the linear system of {\em cbrc}, we use a dense direct method, namely we solve `$u=\mathbf{B}\backslash r$' in MATLAB. Meanwhile, in order to assess the efficiency of $P_{\rm DE1}$,~$P_{\rm DE2}$,~$P_{\rm DE3}$, we solve the linear system~(\ref{10yue5_5}) directly in MATLAB via `$U=H\backslash R$'.

In the tables shown in this section, we use the following notation: `iter' represents the cost of GMRES iterations preconditioned by the $P_{\rm DE1}$, $P_{\rm DE2}$ or $P_{\rm DE3}$ method\footnote{We also try to use the unpreconditioned GMRES for the resulting linear system (\ref{1yue20_1}), but it fails to be convergent and thus the corresponding results are not listed.},
`time' is the CPU elapsed time required for solving the linear system in our {\em B-brm-c} method (either using GMRES preconditioned by $P_{\rm DE1},~P_{\rm DE2}, P_{\rm DE3}$, or using a dense direct method $U=H\backslash R$) and in the {\em cbrc} method (using a direct solver $u=\mathbf{B}\backslash r$), `err' denotes the infinite norm of the absolute errors between the exact and numerical solutions, and finally `order' represents the practical convergence order, which is computed as $log(E_{\frac{M}{2}}/E_{M})/log(2)$.

\textbf{Example 1.} $u'(t)+u(t)=f(t), ~t\in(0,1]$ with the initial value $u(0)=1$, and exact solution $u(t)=e^t$. The purpose of providing this Example 1 is only to verify experimentally the theoretical convergence accuracy of $O\left((2\pi)^{-N}\right)$ proved in Theorem~\ref{th1}.

The results are reported in Table~\ref{tab1}, from which we can see that the error bound of the pure {\em Bm} approach is $O\left((2\pi)^{-N}\right)$, which agrees with the result of Theorem~\ref{th1}.

\begin{table}[!htp]
	\centering
	\caption{Numerical experiments on Example~1.}
	\begin{tabular}{ccccccccc}
		\hline
&\multicolumn{2}{c}{{\em Bm}} \\
		\cline{2-3} \\[-3pt]
		$N$  &\text{err}& $(2\pi)^{-N}$ \\
		\hline
		6& 8.8564e-6&1.6253e-5 \\
		8&2.3609e-7&4.1168e-7 \\
		10&5.9633e-9&1.0428e-8\\
		12&1.5150e-10&2.6414e-10 \\
		\hline
	\end{tabular}
	\label{tab1}
\end{table}

\textbf{Example 2. (Heat equation)} $\frac{\partial u}{\partial t}- \Delta u=f(x,y,t), ~t\in(0,1], ~\Omega=(0,1)\times (0,1)$, with initial value $u(0,x,y)=e^{x+y}$, and exact solution $u(x,y,t)=e^{x+y+t}$.

\textbf{Example 3. (Advection-diffusion equation with variable coefficients)} $\frac{\partial u}{\partial t}+(xy,\sin x\cos y)\cdot \nabla u-e^{x+y}\Delta u=f(x,y,t), ~t\in(0,1],~\Omega=(0,1)\times (0,1),$ with initial value $u(0,x,y)=e^{x+y}$, and exact solution $u(x,y,t)=e^{x+y+t}$. Here $\nabla,~\Delta$ express the gradient operator and Laplacian, respectively.
\begin{table}[!htp]\small
	\centering
	\caption{Numerical experiments on Example~2 ($N=12$).}
	\begin{tabular}{ccccccccccccc}
		\hline
&\multicolumn{4}{c}{{\em B-brm-c}} & &\multicolumn{3}{c}{{\em B-brm-c}} &&
\multicolumn{2}{c}{{\em cbrc}}\\
&\multicolumn{4}{c}{(using $P_{\rm DE1}$ and GMRES)} & &\multicolumn{3}{c}{(using $H\backslash R$)} && \multicolumn{2}{c}{(using $\mathbf{B}\backslash r$)}\\
		\cline{2-5} \cline{7-9}\cline{11-12}\\[-3pt]
		$h$ &\text{iter} &\text{time(s)} &\text{err}&\text{order} & &\text{time(s)}
		&\text{err}&\text{order}&  &\text{time(s)} &\text{err}\\
		\hline
	1/6&5&0.0046&1.1687e-4&--&&0.0031 &1.1687e-4&--&&0.0045&1.1810e-4\\
	1/12&4&0.0048&4.6889e-6& 4.6395&&0.0243&4.6890e-6&4.6395 && 0.0856& 5.1645e-6\\
	1/24&3&0.0225&1.6422e-7&4.8355 && 0.7288&1.6422e-7&4.8356 && 3.0622&2.9508e-6\\
	1/48&3&0.4092& 5.2713e-9&4.9613 &&27.9032 &5.2942e-9&4.9551 &&266.1696&2.9574e-6\\
		\hline
	\end{tabular}
	\label{tab2}
\end{table}

\begin{table}[!htp]
	\centering
	\caption{Numerical experiments on Example~2 ($h=1/50$).}
	\begin{tabular}{cccccccccccccc}
		\hline
&\multicolumn{3}{c}{{\em B-brm-c}} & &\multicolumn{2}{c}{{\em B-brm-c}} &&
\multicolumn{2}{c}{{\em cbrc}}\\
&\multicolumn{3}{c}{(using $P_{\rm DE1}$ and GMRES)} & &\multicolumn{2}{c}{(using $H\backslash R$)} && \multicolumn{2}{c}{(using $\mathbf{B}\backslash r$)}\\
		\cline{2-4} \cline{6-7}\cline{9-10}\\[-3pt]
		$N$ &\text{iter} &\text{time(s)} &\text{err}  & &\text{time(s)}
		&\text{err} &  &\text{time(s)} &\text{err} \\
		\hline
		6&5&0.5351&5.8414e-5&&10.4763 & 5.8414e-5&&23.0434&2.4991e-5\\
		8&5&0.5553&1.5617e-6&& 17.0410&1.5617e-6&&77.9061&1.0303e-5\\
		10&3&0.4795&4.3458e-8&& 26.2615 &4.3456e-8&&221.2240&5.2053e-6\\
		12&3& 0.4984& 4.3525e-9&&36.1127 &4.3710e-9&&493.0315& 2.9579e-6\\
		\hline
	\end{tabular}
	\label{tab3}
\end{table}

\begin{table}[!htp]
	\centering
	\caption{Numerical experiments on Example~3 ($N=12$).}
		\begin{tabular}{ccccccccccc}
		\hline
&\multicolumn{4}{c}{{\em B-brm-c}  } & &\multicolumn{3}{c}{{\em B-brm-c}} \\
&\multicolumn{4}{c}{(using $P_{\rm DE1}$ and GMRES)} & &\multicolumn{3}{c}{(using $H\backslash R$)} \\
		\cline{2-5} \cline{7-9}\\[-3pt]
		$h$ &\text{iter} &\text{time(s)} &\text{err}&\text{order} & &\text{time(s)}
		&\text{err}&\text{order} \\
		\hline
		1/6&4&0.0040&  1.1481e-4&--&&0.0174 &   1.1481e-4&--\\
		1/12&3&0.0043& 4.6539e-6&4.6247 & &0.0294&4.6539e-6& 4.6247\\
		1/24&3&0.0448& 1.6233e-7&4.8414 & &  0.7126& 1.6235e-7&4.8413 \\
		1/48&3&0.5916& 5.4265e-9&4.9028 & &28.4933& 5.4384e-9&4.8998 \\
		\hline
	\end{tabular}
	\label{tab4}
\end{table}

\begin{table}[!htp]
	\centering
	\caption{Numerical experiments on Example~3 ($h=1/50$).}
		\begin{tabular}{cccccccccccc}
		\hline
&\multicolumn{3}{c}{{\em B-brm-c}  } & &\multicolumn{2}{c}{{\em B-brm-c}} \\
&\multicolumn{3}{c}{(using $P_{\rm DE1}$ and GMRES)} & &\multicolumn{2}{c}{(using $H\backslash R$)}\\
		\cline{2-4} \cline{6-7}\\[-3pt]
		$N$ &\text{iter} &\text{time(s)} &\text{err}  & &\text{time(s)}
		&\text{err}  \\
		\hline
		6&4&0.7292 &1.1583e-4&&10.4150&1.1583e-4\\
		8&4&0.7730 &3.1666e-6&&17.7299&3.1666e-6\\
		10&3&0.6558& 8.7396e-8&&25.7434&8.7404e-8\\
		12&3& 0.6968&4.8150e-9&&36.8348&4.8340e-9\\
		\hline
	\end{tabular}
	\label{tab5}
\end{table}

%
%
%
%
%
%
%

The results are reported in Tables~\ref{tab2}--\ref{tab5}, from which we can draw the following conclusions.
\begin{enumerate}
\item From Table~\ref{tab2} and Table~\ref{tab4}, it is evident that the {\em B-brm-c} method is efficient for solving Eq.~(\ref{9yue26_2}) as $\beta_1=0,~\beta_2=1$. It can attain a high convergence order $O(h^{d})$ at the mesh nodes in space.
\item The results in the last column of
	Tables~\ref{tab2}-\ref{tab3} reveal that the {\em cbrc} method has difficulty attaining the same level of accuracy of the {\em B-brm-c} when the same mesh size is used both in time and space. In fact, as shown in Table~\ref{tab2}, the refinement of the mesh does not significantly affect the accuracy of {\em cbrc}. This can be attributed to the coarse time grid used ($N=12$).
\item In terms of CPU time, the {\em cbrc} approach is more expensive than the {\em B-brm-c} method to achieve the same error levels, as shown in Table~\ref{tab3} (the case $N=6$). Additionally, using the GMRES method with the preconditioner $P_{\rm DE1}$ in the {\em B-brm-c} method can save a significant amount of CPU time when compared to a sparse direct solver. This is illustrated in Tables~\ref{tab2}--\ref{tab5}.
\item The results presented in Tables~\ref{tab2}--\ref{tab5} indicate that the preconditioner $P_{\rm DE1}$ is numerically robust and effective to control the number of iterations, as the involved number of iterations almost does not suffer the discretized grid sizes.
\end{enumerate}


\textbf{Example 4. (Wave equation)} $\frac{\partial^2 u}{\partial t^2}-\Delta u=f(x,y,t),~t\in(0,1],~\Omega=(0,1)\times (0,1)$, with initial value $u(0,x,y)=e^{x+y}, \frac{\partial u(0,x,y)}{\partial t}=e^{x+y}$, and exact solution $u(x,y,t)=e^{x+y+t}$.

\begin{table}[!htp]
	\centering
	\caption{Numerical experiments on Example~4 ($N=12$).}
		\begin{tabular}{ccccccccccc}
		\hline
&\multicolumn{4}{c}{{\em B-brm-c}} & &\multicolumn{3}{c}{{\em B-brm-c}} \\
&\multicolumn{4}{c}{(using $P_{\rm DE2}$ and GMRES)} & &\multicolumn{3}{c}{(using $H\backslash R$)} \\
		\cline{2-5} \cline{7-9}\\[-3pt]
		$h$ &\text{iter} &\text{time(s)} &\text{err}&\text{order} & &\text{time(s)}
		&\text{err}&\text{order} \\
		\hline
		1/6&12 & 0.0109 & 1.1652e-4   &-- && 0.0020  & 1.1652e-4  & --\\
        1/12&11 &0.0186 & 4.6856e-6  &4.6362 &&  0.0253  &4.6856e-6   &4.6362 \\
        1/24& 10& 0.0964 & 1.6420e-7  &4.8347 && 0.7924 &1.6420e-7   &4.8347 \\
        1/48& 10&1.2028 &5.2673e-9  &4.9622 &&  31.3442 & 5.2778e-9  &4.9594 \\
		\hline
	\end{tabular}
	\label{tab6}
\end{table}

\begin{table}[!htp]
	\centering
	\caption{Numerical experiments on Example~4 ($h=1/50$).}
		\begin{tabular}{cccccccccccc}
		\hline
&\multicolumn{3}{c}{{\em B-brm-c}} & &\multicolumn{2}{c}{{\em B-brm-c}} \\
&\multicolumn{3}{c}{(using $P_{\rm DE2}$ and GMRES)} & &\multicolumn{2}{c}{(using $H\backslash R$)}\\
		\cline{2-4} \cline{6-7}\\[-3pt]
		$N$ &\text{iter} &\text{time(s)} &\text{err}  & &\text{time(s)}
		&\text{err}  \\
		\hline
		6&14&1.3715& 5.8383e-5 &&11.5031&5.8383e-5\\
        8&13&1.4387&  1.5620e-6&&19.3687&1.5620e-6\\
        10&11& 1.4277&4.3471e-8&&29.1609&4.3474e-8\\
        12&10&1.4308&4.3539e-9&&40.3707&4.3587e-9\\
		\hline
	\end{tabular}
	\label{tab7}
\end{table}

The results of our experiments,  reported in Tables~\ref{tab6}-\ref{tab7}, confirm the conclusions presented in Example~3.
Based on the results shown in Table~\ref{tab6}, it is evident that the {\em B-brm-c} method remains effective for solving Eq.~(\ref{9yue26_2}) when  $\beta_1=1,~\beta_2=0$, and a high convergence order $O(h^{d})$ can also be expected at the mesh nodes in space.
The preconditioner $P_{\rm DE2}$ is robust and numerically scalable, as the number of iterations required to converge is nearly independent of the grid sizes, as shown in Tables~\ref{tab6}-\ref{tab7}. It is evident, when comparing $P_{\rm DE2}$ to $P_{\rm DE1}$, that $P_{\rm DE2}$ requires more time for the same mesh step used. This can be explained by the larger number of iterations involved and by the need for an additional LU factorization of $\theta_4Q+\theta_3I$ when solving the subsystem $\hat{\Delta}_1^{-1}r$ in $P_{\rm DE2}$ according to Theorem~\ref{th7}, as mentioned in Remark~\ref{rem4.7x}.

\textbf{Example 5. (Telegraph equations)} $\frac{\partial^2 u}{\partial t^2}+2\frac{\partial u}{\partial t}-\Delta u+u=f(x,y,t),~t\in(0,1],~\Omega=(0,1)\times (0,1)$, with initial values $u(0,x,y)=e^{x+y}, \frac{\partial u(0,x,y)}{\partial t}=e^{x+y}$, and exact solution $u(x,y,t)=e^{x+y+t}$.

\begin{table}[!htp]
	\centering
	\caption{Numerical experiments on Example~5 ($N=12$).}
		\begin{tabular}{ccccccccccc}
		\hline
&\multicolumn{4}{c}{{\em B-brm-c}} & &\multicolumn{3}{c}{{\em B-brm-c}} \\
&\multicolumn{4}{c}{(using $P_{\rm DE3}$ and GMRES)} & &\multicolumn{3}{c}{(using $H\backslash R$)} \\
		\cline{2-5} \cline{7-9}\\[-3pt]
		$h$ &\text{iter} &\text{time(s)} &\text{err}&\text{order} & &\text{time(s)}
		&\text{err}&\text{order} \\
		\hline
		1/6 & 13 & 0.0246&  1.1124e-4   &-- && 0.0018    &  1.1124e-4  & --\\
        1/12& 12 & 0.0469&  4.6308e-6   &4.5863  &&    0.0341&   4.6308e-6 & 4.5863  \\
        1/24& 11 & 0.1459&  1.6414e-7   &4.8183   &&   1.3311 &  1.6406e-7  &4.8190  \\
        1/48& 11 & 2.0707&  5.7728e-9   &4.8295  &&    78.4841&  5.7490e-9   & 4.8348 \\
		\hline
	\end{tabular}
	\label{tab8}
\end{table}

\begin{table}[!htp]
	\centering
	\caption{Numerical experiments on Example~5 ($h=1/50$).}
		\begin{tabular}{cccccccccccc}
		\hline
&\multicolumn{3}{c}{{\em B-brm-c}} & &\multicolumn{2}{c}{{\em B-brm-c}} \\
&\multicolumn{3}{c}{(using $P_{\rm DE3}$ and GMRES)} & &\multicolumn{2}{c}{(using $H\backslash R$)}\\
		\cline{2-4} \cline{6-7}\\[-3pt]
		$N$ &\text{iter} &\text{time(s)} &\text{err}  & &\text{time(s)}
		&\text{err}  \\
		\hline
		6&16&  2.2355& 6.3320e-5  &&18.1926 & 6.3320e-5\\
        8& 14&2.3216 & 1.5633e-6  &&34.3308 &1.5633e-6 \\
        10&13& 2.2473 &4.4679e-8 &&48.3658 &4.4683e-8 \\
        12&11& 2.2332&4.8183e-9 && 72.9181&4.7658e-9 \\
		\hline
	\end{tabular}
	\label{tab9}
\end{table}

The results are presented in Tables~\ref{tab8}-\ref{tab9}, and analogous conclusions can be deduced from them as in the case of $P_{\rm DE2}$. Table~\ref{tab8} shows that the {\em B-brm-c} method is efficient for Eq.~ (\ref{9yue26_2}) when  $\beta_1\neq 0,~\beta_2\neq 0$, and a high convergence order $O(h^{d})$ at the mesh nodes in space may also be attained. The results presented in Tables~\ref{tab8}-\ref{tab9} indicate that the preconditioner $P_{\rm DE3}$ is effective to control the number of GMRES iterations and is numerically scalable, as we can observe that the number of iterations is nearly independent of the discretized grid sizes. When comparing $P_{\rm DE3}$ to $P_{\rm DE1}$, we can see that $P_{\rm DE3}$ costs more time for the same mesh step used. As before, this can be explained by the larger number of iterations involved and the need for three LU factorizations of $Q, ~\tilde{\Theta}_1,~\tilde{\Theta}_3$ when solving the subsystem $\tilde{\Delta}_1^{-1}r$ in $P_{\rm DE3}$ according to Theorem~\ref{th8}, as mentioned in Remark~\ref{rem4.7x}.

\section{Conclusions}
\label{sec:conclusions}
This paper investigates a new matrix collocation method that combines BRIs and BPs for evolutionary PDEs (\ref{9yue26_2}). The numerical results indicate that this approach can significantly enhance both CPU time efficiency and approximating accuracy when compared to the more conventional {\em crbc} method. Then, in order to efficiently solve the resulting system of linear equations, we introduced three dimension expanded preconditioners denoted as $P_{\rm DE1},~P_{\rm DE2},~P_{\rm DE3}$ that take advantage of the structural properties of the coefficient matrix $H$ in~(\ref{10yue5_5}). The accuracy and CPU time efficiency of the presented {\em B-brm-c} method have been established through both numerical examples and theoretical analyses. Our experiments show that the proposed preconditioners $P_{\rm DE1},~P_{\rm DE2},~P_{\rm DE3}$ are numerically robust and effective to control the number of iterations.

Our future work for this {\em B-brm-c} method includes its extension to the solution of nonlinear evolutionary PDEs with non-smooth initial-boundary conditions. Additionally, we aim to investigate the associated preconditioning techniques that are applicable to this problem.


\section*{Conflict of interest}
The authors declare that they have no conflict of interest.

\end{document}